\documentclass[aos]{imsart}
\usepackage{geometry}                % See geometry.pdf to learn the layout options. There are lots.
\RequirePackage{natbib}
\usepackage{graphicx}
\usepackage{amsthm}
\usepackage{amssymb}
\usepackage{epstopdf}
\usepackage[rightcaption]{sidecap}
\usepackage{subfigure}
\usepackage{ulem}
\usepackage{color}
\newtheorem{definition}{Definition}
\newtheorem{theorem}{Theorem}
\newtheorem{corollary}{Corollary}
\newtheorem{lemma}{Lemma}
\newtheorem{prop}{Proposition}

\startlocaldefs

\newcommand{\pre}{\ensuremath{F}}

\newcommand{\R}{\ensuremath{{\mathbb{R}}}}
\newcommand{\X}{\ensuremath{{\mathbf{X}}}}

\newcommand{\Y}{\ensuremath{{Y}}}
\newcommand{\e}{\ensuremath{{\epsilon}}}

\newcommand{\minbeta}{\ensuremath{{M(\truebeta)}}}
\newcommand{\Xa}{\ensuremath{{X(S)}}}
\newcommand{\Xb}{\ensuremath{{X(S^c)}}}

\newcommand{\truebeta}{\ensuremath{{\beta^*}}}
\newcommand{\truebetaj}{\ensuremath{{\beta_j^*}}}

\newcommand{\betA}{ {\truebeta(S)}}

\newcommand{\hbet}{\ensuremath{\hat \beta(\lam)}}

\newcommand{\hbetA}{\ensuremath{\hat{\beta}^{(1)}}}

\newcommand{\p}{\ensuremath{{p}}}

\newcommand{\lam}{\ensuremath{{\lambda}}}

\newcommand{\sign}{\ensuremath{{\mbox{sign}}}}
\newcommand{\Prob}[1]{\ensuremath{P\left( #1\right)}}

\endlocaldefs

\begin{document}

\begin{frontmatter}

% "Title of the paper"
\title{Preconditioning to comply with the Irrepresentable Condition}
\runtitle{Preconditioning and the Lasso}

% indicate corresponding author with \corref{}
% \author{\fnms{John} \snm{Smith}\corref{}\ead[label=e1]{smith@foo.com}\thanksref{t1}}
% \thankstext{t1}{Thanks to somebody} 
% \address{line 1\\ line 2\\ printead{e1}}
% \affiliation{Some University}

\author{\fnms{Jinzhu} \snm{Jia}\ead[label=e1]{jzjia@math.pku.edu.cn}\thanksref{t1}\thanksref{t2}}
\address{LMAM School of Mathematical Sciences \\
         and Center for Statistical Science\\
         Peking University\\
         Beijing 100871, China \\
\printead{e1}}
\affiliation{Peking University}

\and
\author{\fnms{Karl} \snm{Rohe}\ead[label=e2]{karlrohe@stat.wisc.edu}\thanksref{t1}\thanksref{t2}}
\address{Department of Statistics\\ University of Wisconsin-Madison\\ WI 53706, USA\\
\printead{e2}}
\affiliation{University of Wisconsin--Madison}
\thankstext{t1}{The authors contributed equally and are listed in alphabetical order.}
\thankstext{t2}{Thank you to Bin Yu and Peter Qian for fruitful conversations while this idea was still unformed.   Important steps of this research were conducted at the Institute for Mathematics and its Applications' (IMA's) conference ``Machine Learning: Theory and Computation."  Thank you to IMA for both the conference and the financial support.  Thank you to Nicolai Meinshausen for sharing code that contributed to Section \ref{geo}.  Thank you to Sara Fernandes-Taylor for helpful comments.} 

\runauthor{Jia, Rohe}

\begin{abstract}
Preconditioning is a technique from numerical linear algebra that can accelerate algorithms  to solve systems of equations.  %Preconditioning has not yet been applied for statistical inference.  
In this paper, we demonstrate how preconditioning can circumvent a stringent assumption for sign consistency in sparse linear regression. 
% Sparse linear regression techniques, such as the Lasso, SCAD, and MC+, simultaneously select and fit a linear model.  
% Several authors have shown that these techniques requires stringent assumptions for sign consistency.  
Given $\X \in \R^{n \times p}$ and $Y \in \R^n$ that satisfy the standard regression equation, this paper demonstrates that even if the design matrix $\X$ does not satisfy the irrepresentable condition for the Lasso, the design matrix $\pre  \X$ often does, where $\pre  \in \R^{n\times n}$ is a  preconditioning matrix defined in this paper. 
By computing the Lasso on $(\pre  \X, \pre  Y)$, instead of on $(\X, Y)$, the necessary assumptions on $\X$ become much less stringent.   

Our  preconditioner $\pre $ ensures that the singular values of the design matrix are either zero or one. When $n\ge p$,  the columns of $\pre  \X$ are orthogonal and the preconditioner always circumvents the stringent assumptions.  When $p\ge n$, $\pre$  projects the design matrix onto the Stiefel manifold; the rows of $\pre  \X$ are orthogonal.  We give both theoretical results and simulation results to show that, in the high dimensional case, the preconditioner helps to circumvent the stringent assumptions,  improving the statistical performance of a broad class of model selection techniques in linear regression.  Simulation results are particularly promising. 

%For $n>p$, we demonstrate that after the Puffer Transformation it is sensible to apply an uneven penalty in the Lasso.  We give a specific penalty that depends only on the design matrix.  The resulting algorithm is sign consistent and nearly identical to a form of one-step, backwards elimination. 
\end{abstract}

%\begin{keyword}[class=AMS]
%\kwd[Primary ]{}
%\kwd{}
%\kwd[; secondary ]{}
%\end{keyword}
\begin{keyword}
\kwd{Preconditioning}
\kwd{Lasso}
\kwd{Sign consistency}
\end{keyword}

\end{frontmatter}

% AOS,AOAS: If there are supplements please fill:
%\begin{supplement}[id=suppA]
%  \sname{Supplement A}
%  \stitle{Title}
%  \slink[url]{http://lib.stat.cmu.edu/aoas/???/???}
%  \sdescription{Some text}
%\end{supplement}

\section{Introduction}

Recent breakthroughs in information technology have provided new experimental capabilities in astronomy, biology, chemistry, neuroscience, and several other disciplines.  Many of these new measurement devices create data sets with many more ``measurements" than units of observation.  For example, due to experimental constraints, both fMRI and microarray experiments often include tens or hundreds of people.  However, the fMRI and microarray technologies can simultaneously measure 10's or 100's of thousands of different pieces of information for each individual.  Classical statistical inference in such ``high-dimensional" or $p>\hspace{-.06in}>n$ regimes is often impossible.  Successful experiments must rely on some type of sparsity or low-dimensional structure.  Several statistical techniques have been developed to exploit various types of structure, including sparse high dimensional regression.

Sparse high dimensional regression aims to select the few measurements (among the 10's of thousands) that relate to an outcome of interest;  these techniques can screen out the irrelevant variables. A rich theoretical literature describes the consistency of various sparse high dimensional regression techniques, highlighting several potential pitfalls   (e.g. \cite{knight2000asymptotics, fan2001variable, greenshtein2004persistence, donoho2006stable, meinshausen2006high, tropp2006just, zhao2006model, zou2006adaptive, zhang2008sparsity, fan2008sure, wainwright2009, meinshausen2009lasso, bickel2009simultaneous,  zhang2010nearly, shao2012estimation}).   In this literature, one of the most popular measures of asymptotic performance is sign consistency, which implies that the estimator selects the correct set of predictors asymptotically.  One of the most popular methods in sparse regression, the Lasso (defined in Section \ref{prelim}), requires a stringent ``irrepresentable condition" to achieve sign consistency \citep{tibshirani1996regression, zhao2006model}.  The irrepresentable condition restricts the correlation between the columns of the design matrix in a way made explicit in Section \ref{prelim}.

	 It is well known that the Ordinary Least Squares (OLS) estimator performs poorly when the columns of the design matrix are highly correlated.  However, this problem can be overcome by more samples; OLS is still consistent.   
%These techniques provide sparse estimates, enabling various types of statistical consistency even when the number of predictors, or covariates, greatly exceeds the number of predictors.  There is now a sound understanding of the asymptotic behavior of the Lasso under the standard linear model \citep{the papers in bickel's paper}.  
With the Lasso, the detrimental effects of correlation are more severe.  If the columns of the design matrix are correlated in a way that violates the irrepresentable condition, then the Lasso will not be sign consistent (i.e. statistical estimation will not improve with more samples).

To avoid the irrepresentable condition, several researchers have proposed alternative penalized least square methods that use a different penalty from the Lasso penalty.
For example,  \cite{fan2001variable} propose SCAD, a concave penalty function;  \cite{zhang2010nearly} proposes the minimax concave penalty (MCP), another concave penalty function, and gives high probability results for PLUS, an optimization algorithm.  Unfortunately, these concave penalties lead to nonconvex optimization problems.  Although there are algorithmic approximations for these problems and some high probability results \citep{zhang2010nearly}, the estimator that these algorithms compute is not necessarily the estimator that optimizes the penalized least squares objective.  The Adaptive Lasso provides another alternative penalty  which is a data adaptive and heterogeneous \citep{zou2006adaptive}.  Unfortunately, its statistical performance degrades in high dimensions. 

In penalized least squares, there is both a penalty and a data fidelity term (which makes the estimator conform to the data).  The papers cited in the previous paragraph adjust the type of sparse penalty.  In this paper, we precondition the data, which is equivalent to adjusting the data fidelity term.  Other researchers have previously proposed alternative ways of measuring data fidelity \citep{van2008high}, but their alternatives are meant to accommodate different error distributions, not to avoid the irrepresentable condition. 
Similar to work presented here, \cite{xiong2011orthogonalizing} also propose adjusting the data fidelity term to avoid the irrepresentable condition.  They proposed a procedures which (1) makes the design matrix orthogonal by adding rows, and (2) applies an EM algorithm, with SCAD, to estimate the outcomes corresponding to the additional rows in the design matrix.  Although this algorithm performs well in the low dimensional case, it is computationally expensive in high dimensional problems.  The procedure proposed in this paper adjusts the data fidelity term by preconditioning, a preprocessing step.  Relative to these alternative methods, preconditioning is easier to implement,   requiring only a couple lines of code before calling any standard Lasso package.  Furthermore, this type of preprocessing is widely studied in a related field, numerical linear algebra.

Preconditioning describes a popular suite of techniques in numerical linear algebra that stabilize and accelerate  algorithms to solve systems of equations  (e.g. \cite{axelsson1985survey, golub1996matrix}).  In a system of equations, one seeks the vector $x$ that satisfies $Ax=b$, where  $A \in \R^{n \times n}$ is  a given matrix and $b\in \R^n$ is a given vector.  The speed of most solvers is inversely proportional to the condition number of the matrix $A$, the ratio of its largest eigenvalue over its smallest eigenvalue.  When the matrix $A$ has both large eigenvalues and small eigenvalues,  the system $Ax = b$ is  ``ill-conditioned."  For example, if the matrix $A$ has highly correlated columns,  it will be ill-conditioned.  One can ``precondition" the problem by left multiplying the system by a matrix $T$, $TAx=Tb$; the preconditioner $T$  is designed to shrink the condition number of $A$ thereby accelerating the system solver.  If the columns of $A$ are highly correlated, then preconditioning decorrelates the columns. 

The system of equations $Ax = b$ has many similarities with the  linear regression equation 
\begin{equation} \label{regeq}
Y = \X \truebeta + \epsilon, 
\end{equation}
where we observe $Y \in \R^n$ and $\X \in \R^{n \times p}$, and $\epsilon \in \R^n$ contains unobserved iid noise terms with $E(\epsilon) = 0$ and $var(\epsilon)  = \sigma^2 I_n$.
From the system of equations $Ax =b$, the regression equation adds an error term and allows for the design matrix to be rectangular.  
Where numerical linear algebraists use preconditioners for algorithmic speed, %we will show that statisticians can employ preconditioning for improved statistical inference. 
this  paper shows that preconditioning can circumvent the irrepresentable condition, improving the statistical performance of the Lasso. 
%several stringent assumptions in sparse linear regression, improving the statistical performance of a broad class of model selection techniques in linear regression.

%In Ordinary Least Squares (OLS), if the columns of $\X$ are highly correlated, then the OLS estimator has a large variance.  Preconditioning cannot address this problem.  However, in sparse linear regression, correlation presents an entirely different problem.  

%Under the regression equation \eqref{regeq}, 
%   %Linear regression aims to estimate the vector $\truebeta \in \R^p$ as a function of $\X$ and $Y$.   It is of particular interest when $\truebeta$ contains several zeros.  
%if $\truebetai = 0$, then $Y$ is conditionally independent of the $i$th column of $\X$, given all other columns of $\X$.  In recent years, a great deal of attention has been paid to penalized least squares techniques that estimate which elements of $\truebeta$ equal zero.  
%XXXX There are several such techniques, for example, Bridge, Lasso, Danzig selector, SCAD, and MC+ \citep{bridge, danzig, scad, mcp}.XXX

		 %and the more recent literature that the problem of model selection is greatly simplified when the design matrix $\X$ contains orthogonal columns.   In linear regression, irrepresentable condition and RIP fail when the design matrix has correlated columns.  
		 Just as preconditioning sidesteps the difficulties presented by correlation in systems of equations, preconditioning can sidestep the difficulties in sparse linear regression.  Numerical algebraists precondition systems of linear equations to make algorithms faster and more stable.  In this paper, we show that preconditioning  the regression equation (Equation \ref{regeq}) can circumvent the irrepresentable condition.   For a design matrix $\X \in \R^{n \times p}$, we study a specific preconditioner $\pre  \in \R^{n \times n}$ that is defined from the singular value decomposition of $\X = UDV'$.  We call $\pre = U D^{-1} U'$ the Puffer Transformation because it inflates the smallest nonsingular values of the design matrix (Section \ref{geo} discusses this in more detail).  
%		Let $\X = UDV^T$ be the singular value decomposition of $\X$.  For $d = \min\{n,p\}$, $U \in \R^{n \times d}$ and $V \in \R^{p \times d}$ are orthonormal matrices and $D\in \R^{d \times d}$ is a diagonal matrix.  Define
%\begin{equation} \label{badgertransform}
%B = UD^{-1}U^T
%\end{equation}
This paper demonstrates why the matrix $\pre \X$ can satisfy the irrepresentable condition, while the matrix $\X$ may not;
%With linear regression and the least squares loss function, the lack of linear dependence between the predictors implies that one can estimate $\truebetai$ and $\truebetaj$ independently.  
in essence, the preconditioner   makes the columns of $\X$ less correlated.  
%Preconditioning left multiplying the regression equation (Equation \ref{regeq}) by a transformation $T$ which is a function of $\X$.  We introduce a specific transformation $\pre \in \R^{n \times n}$ which, 
When $n\ge p$,  the columns of $\pre \X$ are exactly orthogonal, trivially satisfying the irrepresentable condition.  When $n<p$ and the columns of $\X$ are moderately or highly correlated,  $\pre $ can greatly reduce the the pairwise correlations between columns, making the design matrix $\X$ more amenable to the irrepresentable condition.
		
In a paper titled \textit{``Preconditioning" for feature selection and regression in high-dimensional problems}, the authors propose  projecting the outcome $Y$ onto the top singular vectors of $\X$ before running the Lasso \citep{paul2008preconditioning}.  They leave $\X$ unchanged.  
%This method could potentially be confused with the method described in ,  In their preprocessing step, 
%They note that ``when the observational noise is rather large, the suggested procedure can give a more accurate estimate than Lasso."  In this way, their ``preconditioning" is meant to reduce the noise.  
The current paper preconditions the entire regression equation, which reduces the impact of the top singular vectors of $\X$, and thus reduces the correlation between the columns of  $\X$.
Whereas the preprocessing step in \cite{paul2008preconditioning} performs noise reduction, the Puffer Transform makes the design matrix conform to the irrepresentable condition.

The outline of the paper is as follows:  Section \ref{prelim} gives the necessary mathematical notation and definitions.  Section \ref{preconditioning} introduces the Puffer Transformation and gives a geometrical interpretation of the transformation.      Section \ref{lowdim} discusses the low dimensional setting ($p\le n$), where the Puffer Transformation makes the columns of $\pre \X$ orthogonal.  Theorem \ref{thm:fixeddim} gives sufficient conditions for the sign consistency of the preconditioned Lasso when $p\le n$;  these sufficient conditions do not include an irrepresentable condition. 
%This simultaneously makes $\pre \X$ satisfy the Lasso's conditions and makes the SCAD and MC+ estimators computationally tractable;  since the columns of $\pre \X$ are orthogonal, the Lasso, SCAD, and MC+ all admit closed form solutions.   
%For $n>p$, Theorem \ref{thm:fixeddim} gives sufficient conditions for the Lasso to be sign consistent.  These conditions only require a lower bound on the smallest singular value of $\X$.  The conditions do \textit{not} include an irrepresentable condition or RIP like assumption.   
Section \ref{highdim} discusses the high dimensional setting ($p>n$), where the Puffer Transformation projects the design matrix onto the Stiefel manifold.  Theorem \ref{unifTheorem} shows that most matrices on the Stiefel manifold satisfy the irrepresentable condition.  Theorem \ref{highdimthm} gives sufficient conditions for the sign consistency of the preconditioned Lasso in high dimensions.  This theorem includes an irrepresentable condition.   Section \ref{sim} shows promising simulations that compare the preconditioned Lasso to several other (un-preconditioned) methods.  Section \ref{related} describes four data analysis techniques that incidentally precondition the design matrix with a potentially harmful preconditioner; just as a good preconditioner can improve estimation performance, a bad preconditioner can severely detract from performance.   Users should be cautious when using the four techniques described in Section \ref{related}.  Section \ref{disc} concludes the paper.

%		  ``Incoherence" (InC) or
%
%	
%   In several settings, there are reasons to believe that the vector $\truebeta$ contains several zeros.  In effect, several of the columns of $X$ are irrelevant to the outcome $Y$.  The Lasso, SCAD, and MC+ all bias their estimators so that several several elements are exactly zero. In effect, these methods select the relevant  measurements and the irrelevant measurements in $X$.  The RIP, InC, and irrepresentable condition all make requirements on submatrices of $X$.  In effect, these conditions fail when the columns of $X$ are too correlated. 
%
%
%
%		  This paper proposes preconditioning for statistical inference.
%		
%		 , where one is given a matrix $p$
%
%		low dim.  high dim: some theory and simulations.  Data analytic techniques to complement the preprocessing.
%		
%		 
%	picture of diamond intersecting ellipse and diamond intersecting circle.  + simulation with n=200, p-->50k'

%\subsection{Overview of previous work}

\subsection{Preliminaries} \label{prelim}
To define the Lasso estimator, suppose the observed data  are independent pairs
$\{( x_i,\Y_i)\} \in \R^p \times  \R$ for $i = 1, 2,  \dots , n$
following the linear regression model
\begin{equation}
\Y_i = x_i^T\truebeta + \e_i , \label{linear}
\end{equation}
where $x_i^T$ is a row vector representing the predictors for the
$i$th observation, $Y_i$ is the corresponding $i$th response
variable, and $\e_i$'s are independent, mean zero noise terms with variance $\sigma^2$.  The unobserved coefficients are $\truebeta \in \R^p$. Use $\X \in \R^{n\times p}$ to denote the $n\times p$ design matrix with
$x_k^T=\left(\X_{k1},\ldots,\X_{kp}\right)$ as its $k$th row and
with $X_j=\left(\X_{j1},\ldots,\X_{jn}\right)^T$ as its $j$th
column, then
\[ \X= \left( \begin{array}{c}
x_1^T \\
x_2^T\\
\vdots \\
x_n^T
\end{array} \right)= \left(X_1,X_2,\ldots, X_p\right).\]
Let $Y=\left(Y_1,\ldots,Y_n\right)^T$ and
$\e=\left(\e_1,\e_2,\ldots, \e_n\right)^T\in \R^n$. 

For penalty function $pen(b) : \R^p \rightarrow \R$, define the penalized least squares objective function,
\begin{equation} \label{optprob}
\ell(b, pen, \lambda) = \frac{1}{2}\|Y - \X b\|_2^2 + pen(b, \lambda).
\end{equation}
The Lasso estimator uses the $\ell_1$ penalty,
\begin{equation} \label{lasso}
\hbet = \arg \min_b \frac{1}{2}\|Y - \X b\|_2^2 + \lambda \ \| b \|_1.
\end{equation}
%
% then defined as the solution to a penalized least
%squares problem (with regularization parameter $\lambda$):
%
%\begin{equation}
%\hbet =\arg\min_{\bet} \|\Y-\X\bet\|_2^2+\lam
%\|\bet\|_{1}\label{lasso},
%\end{equation}
where for some vector $b \in \R^p$, $\|b\|_r = (\sum_{i=1}^k
|x_i|^r)^{1/r}$.  
%Both SCAD and MC+ are solutions to penalized least squares problems, where the penalty function is a concave penalty.  For SCAD the penalty is given by
%\[\ pen(b, \lambda) = \left\{\begin{array}{ll} \lambda |b|, & |b|\leq \lambda,\\
%-(b^2-2a\lambda|b|+\lambda^2)/[2(a-1)], & \lambda <|b|\leq a\lambda,\\
%(a+1)\lambda^2/2, & |b|>a\lambda,
%\end{array}
%\right.\]
%for some $a>2$.
%
%For MC+, the penalty is given by
%\[ \ pen(b, \lambda)  = \lambda\int_{0}^{|b|}\left(1-\frac{x}{\gamma \lambda}\right)^+ dx,\]
%for some $\gamma>0$.

The popularity of the Lasso (and other sparse penalized least squares methods)  stems from the fact that, for large enough values of $\lam$, the estimated coefficient vectors contain several zeros.  If one is willing to assume the linear regression model, then the Lasso estimates which columns in $\X$  are conditionally independent of $Y$ given the other columns in $\X$.

%In research that proceeded the Lasso, \cite{chen1994basis} proposed the Basis Pursuit algorithm for the noiseless setting ($\epsilon = 0$):
%\begin{equation} \label{BP}
%\hat \beta_{BP} = \arg \min_{b: \ Y= X b} \|b\|_1
%\end{equation}
%This objective function is similar to the Lasso.  \cite{candes2005decoding} showed that under certain conditions, $\hat \beta_{BP}$ is equal to the solution to the NP hard problem {\color{red}(I have to check again about this claim.)}
%\[\min_{b: \ Y = Xb} \|b\|_0,\]
%where $\|b\|_0 = |\{i: \ b_i \ne 0\}|$.  

\subsubsection{Sign consistency and the irrepresentable condition} \label{previousassumptions}

%The irrepresentable condition is senstive to the corr The theorems in the literature  on sign consistency and signal recovery rely on different assumptions.  However,  their assumptions are similar in that they measure the correlation among the columns of $\X$.  When the columns of $\X$ are highly correlated, the Puffer Transformation decreases the correlations.  As such, the Puffer Transformation can simultaneously make both assumptions more attainable.  This subsection defines the two types of assumptions and {\color{red} explains} how they are related. 

%First, some notation.  
For $T \subset \{1, \dots, p\}$ with $|T| = t$,  define $\X(T) \in \R^{n \times t}$ to contain the columns of $\X$ indexed by $T$.  For any vector $x \in \R^p$, define $x(T) = (x_j)_{j \in T}$. $S \subset \{1, \dots, p\}$, the support of $\truebeta$,  is defined
\[S = \{j : \truebetaj \ne 0\}.\]
Define $s = |S|$.  In order to define sign consistency, define 
\[\textrm{sign}(x) =  \left\{\begin{array}{rl}1 & \textrm{ if } x>0 \\0 & \textrm{ if } x=0 \\-1 & \textrm{ if } x<0,\end{array}\right.\]
 and for a vector $b$, $\textrm{sign}(b)$ is defined as a vector with the $i$th element $\textrm{sign}(b_i)$.

\begin{definition} 
The Lasso is \textbf{sign consistent} if there exists a sequence
$\lambda_n$ such that,
\[P\left(\textrm{sign}(\hat\beta(\lambda_n)) =  \textrm{sign}(\truebeta )\right) \rightarrow 1,  \mbox{ as } n \rightarrow \infty.\]
\end{definition}
In other words,   $\hbet$ can asymptotically identify the relevant and irrelevant  variables when it is sign consistent. Several authors, including 
\cite{meinshausen2006high,zou2006adaptive, zhao2006model, yuan2007non}, have studied the sign consistency property and found a sufficient condition for sign consistency.  \cite{zhao2006model} called this assumption the  ``irrepresentable condition"  and showed that it is almost necessary for sign consistency.  
%To define the irrepresentable condition,  without loss of generality, assume that the first $s$ elements of $\truebeta$ are the only nonzero elements.  Define $\Xa \in \R^{n \times s}$ as the first $s$ columns of $\X$ (these are the relevant predictors) and define $\Xb \in \R^{n \times (p-s)}$ as the last $p-s$ columns of $\X$ (these are the irrelevant predictors).  For a vDefine $\betA = (\truebetaj)_{j \in S}$ and $\overrightarrow{b} = \textrm{sign} (\betA).$  
%define $S = \{i : \truebeta_i \ne 0\}$ and $\Xa \in \R^{n \times s}$ whose columns are $\{X_i: i \in S\}$.  $\Xa$ contains the relevant columns of $\X$ and $\Xb$ contains all the rest. 
\begin{definition}
The design matrix $\X$ satisfies the \textbf{Irrepresentable condition} for $\truebeta$ if, for some constant $\eta \in (0,1]$,
\begin{equation}
\left\|\Xb^T\Xa\left(\Xa^T\Xa\right)^{-1}sign(\truebeta(S))\right\|_{\infty}\leq
1-\eta, \label{IC}
\end{equation}
where for a vector $x$, $\|x\|_{\infty}  = \max_i |x_i|$.
\end{definition}

In practice, this condition is difficult to check because it relies on the unknown set $S$.  Section 2 of \cite{zhao2006model} gives several sufficient conditions.  For example, their Corollary 2 shows that if $|\mbox{cor}(X_i, X_j)| \le c/(2s - 1)$ for a constant $0\le c <1$, then the irrepresentable condition holds.  Theorem \ref{unifTheorem} in Section \ref{highdim} of this paper relies on their corollary.

\section{Preconditioning to circumvent the stringent assumption} \label{preconditioning}

We will always assume that the design matrix $\X \in \R^{ n \times p}$ has rank $d = \min\{n,p\}$.  From singular value decomposition, there exist matrices $U \in \R^{n \times d}$ and $V\in \R^{p \times d}$  with $U^TU=V^TV = I_d$ and diagonal matrix $D \in \R^{d \times d}  $ such that $\X = UDV'$.  Define the \textbf{Puffer Transformation},
\begin{equation} \label{puffertransform}
\pre  = UD^{-1}U^T.
\end{equation}
The preconditioned design matrix $\pre \X$ has the same singular vectors as $\X$.  However, all of the nonzero singular values of $\pre \X$ are set to unity: $\pre \X = UV'$.  When $n\ge p$, the columns of $\pre \X$ are orthonormal.  When $n \le p$, the rows of $\pre \X$ are orthonormal. 

%Further, for $p\ge n$, $\pre  = (\X\X^T)^{-1/2}$.  So, the matrix $\pre  \X$ is the directional component in the polar decomposition of $\X=  (\X\X^T)^{1/2} (\pre  \X)$.  Under any unitarily invariant norm, the directional component $\pre \X$ is the projection of $\X$ onto the Stiefel manifold, the space of orthonormal matrices \citep{fan1955some}.  Theorem \ref{unifTheorem} will show that most matrices on the Stiefel manifold satisfy the irrepresentable condition so long as $s$ is not too large. 

Define $\tilde Y = \pre Y, \ \tilde \X = \pre \X,$ and $\tilde \e = \pre \e$.   After left multiplying the regression equation $Y = \X \truebeta + \e$ by the matrix $\pre $, the transformed regression equation becomes
\begin{equation} \label{newregeq}
\tilde Y = \tilde X \truebeta + \tilde \e
\end{equation}
If $\e \sim N(0, \sigma^2 I_n)$, then $\tilde \e \sim N(0, \tilde \Sigma)$ where $\tilde \Sigma = \sigma^2 UD^{-2}U^T$. 
%%\begin{eqnarray*}
%%\tilde \Sigma &=& E(\tilde e \tilde e^T) \\
%%&=& \pre  E(\e \e^T) \pre ^T \\
%%&=& \sigma^2 UD^{-1}U^T UD^{-1}U^T\\
%%&=& \sigma^2 UD^{-2}U^T. 
%%\end{eqnarray*}
%The error terms $\tilde \e$ are dependent and the variance of individual $\tilde \e_i$'s potentially exceeds the variance of the individual $\e_i$'s {\color{red} [The above comment might not be true and might meaningless. The variance itself is not meaningful but the signal to noise ratio is. So we do not have to compare $var(\tilde \epsilon_i)$ with $var(\epsilon_i)$.]} This added noise must be weighed against the potential benefits of a well behaved design matrix. {\color{red} [The above comment needs more explanation. ]}  

The scale of $\tilde \Sigma$ depends on the diagonal matrix $D$, which contains the $d$ singular values of $\X$.     As the singular values of $\X$ approach zero, the corresponding elements of $D^{-2}$ grow very quickly.  This increased noise can quickly overwhelm the benefits of a well conditioned design matrix.  For this reason, it might be necessary 
%to threshold the small values of $D$ or 
add a Tikhonov regularization term to the diagonal of $D$.  The simulations in Section \ref{sim} show that when $p \approx n$, the transformation can harm estimation.  Future research will examine if a Tikhonov regularizer resolves this issue. 

	In numerical linear algebra, the objective is speed, and there is a trade off between the time spent computing the preconditioner vs.\ solving the system of equations.  Better preconditioners make the original problem easier to solve.  However, these preconditioners themselves can be time consuming to compute.   Successful preconditioners balance these two costs to provide a computational advantage.  In our setting, the objective is inference, not speed per se, and the tradeoff is between a well behaved design matrix and a well behaved error term.  Preconditioning can aid statistical inference if it can balance these two constraints.

\subsection{Geometrical representation} \label{geo}

The figures in this section display the geometry of the Lasso before and after the Puffer Transformation.  These figures (a) demonstrate what happens when the irrepresentable condition is not satisfied, (b) reveal how the Puffer Transformation circumvents the irrepresentable condition, and (c) illustrate why we call $\pre$ the Puffer Transformation.

The figures in this section are derived from the following optimization problem which is equivalent to the Lasso.\footnote{In an abuse of notation, the left hand side of Equation \ref{constrainedlasso} is $\hat \beta(c)$.  In fact, there is a one-to-one function $\phi(c) = \lambda$ to make the Lagrangian form of the Lasso (Equation \ref{lasso}) equivalent to the constrained form of the Lasso (Equation \ref{constrainedlasso}). }
\begin{equation}\label{constrainedlasso}
\hat \beta(c) = \arg \min_{b: \|b\|_1  \leq c} \|Y - \X b\|_2^2
\end{equation}

Given the constraint set $\|b\|_1 \le c$ and a continuum of sets $\|Y - X b\|_2^2 \le x$ for $x\ge 0$, define 
\[\mathcal{I}(c,x) = \{b: \|b\|_1 \le c\} \cap \{b: \|Y - X b\|_2^2 \le x\}.\]
When $c$ is small enough, $\mathcal{I}(c,0)$ is an empty set, implying that there is no $b$ with $\|b\|_1 \le c$ such that $Y=Xb$.  To find $\hat \beta(c)$, increase the value of $x$ until $\mathcal{I}(c,x)$ is no longer an empty set.  Let $x^*$ be the smallest $x$ such that $\mathcal{I}(c,x)$ is nonempty.  Then, $\hat \beta(c) \in \mathcal{I}(c, x^*)$.  Under certain conditions on $X$ (e.g. full column rank), the solution is unique and $\hat \beta(c) = \mathcal{I}(c, x^*)$.  

Figures \ref{fig:geocordesign} and \ref{fig:geopreconditioned} below give a graphical representation of this description of the Lasso before and after preconditioning. 
%In Figures \ref{fig:geocordesign} and \ref{fig:geopreconditioned} below, 
The constraint set $\{b: \|b\|_1 \le c\}$ appears as a diamond shaped polyhedron and the level set of the loss function $\{b: \|Y - X b\|_2^2 < x\}$ appears as an ellipse.   Starting from $x = 0$, $x$ increases, dilating the ellipse, until the ellipse intersects the constraint set.  The first point of intersection represents the solution to the Lasso.  This point $\hat \beta(c) \in \R^p$ is the element of the constraint set which minimizes $\|Y - X b\|_2^2$. 

In both Figure \ref{fig:geocordesign} and Figure \ref{fig:geopreconditioned}, the rows of $X$ are independent Gaussian vectors with mean zero and covariance matrix
\[\Sigma = \left(\begin{array}{ccc}1 & 0 & .6 \\0 & 1 & .6 \\ .6 & .6 & 1\end{array}\right)\]
To highlight the effects of preconditioning, the noise is very small and $n = 10,000$.  In Figure \ref{fig:geocordesign}, the design matrix is not preconditioned.  In Figure  \ref{fig:geopreconditioned}, the problem has been preconditioned, and the ellipse represents the set $\|\pre Y - \pre X b\|_2^2 \le x$ for some value of $x$.  The preconditioning turns the ellipse into a sphere.

\begin{figure}
\centering
\mbox{\subfigure{\includegraphics[width=2.5in]{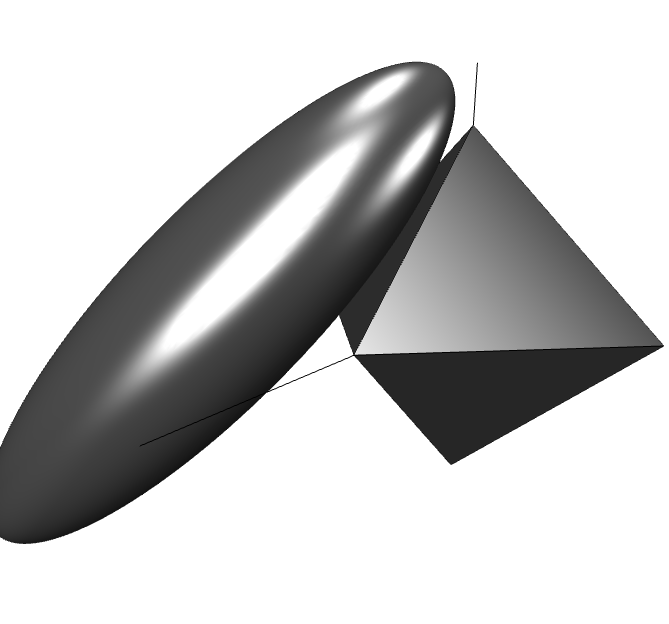}}\quad
\subfigure{\includegraphics[width=2.5in]{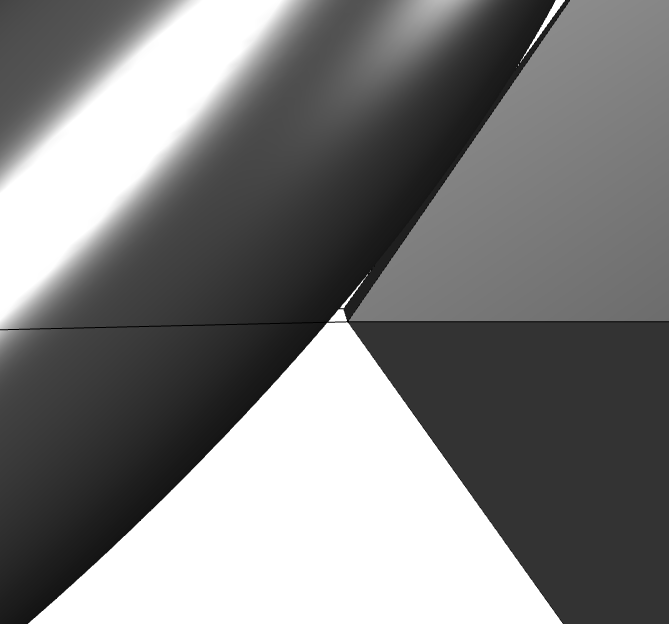} }}
\caption{When the problem is not preconditioned, the elongated ellipse (representing the $\ell_2$ loss) intersects the $\ell_1$ ball off of the plane created from the axes that point to the left.  The Lasso fails to select the true model.} 
\label{fig:geocordesign}
\end{figure}

%\begin{figure}[htbp] %  figure placement: here, top, bottom, or page
%   \centering
%\includegraphics[height = 3in, width = 3in]{Simulations/correlatedDesign.png} 
%   \caption{When the problem is not preconditioned, the ellipse that represents $\ell_2$ loss is elongated.  Because of the elongation, it intersects the $\ell_1$ ball off of the plane created from the axes that point to the left.  Without preconditioning, the Lasso fails to select the true model.}
%   \label{fig:geocordesign}
%\end{figure}

\begin{figure}
\centering
\mbox{\subfigure{\includegraphics[width=2.5in]{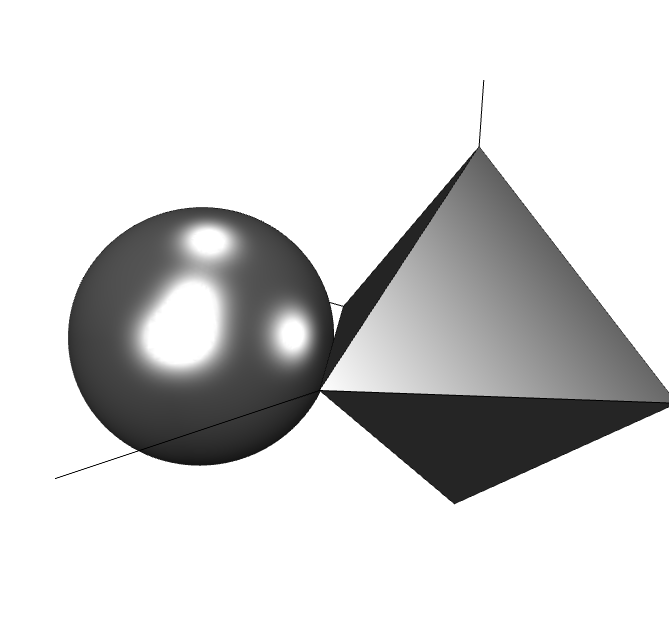}}\quad
\subfigure{\includegraphics[width=2.5in]{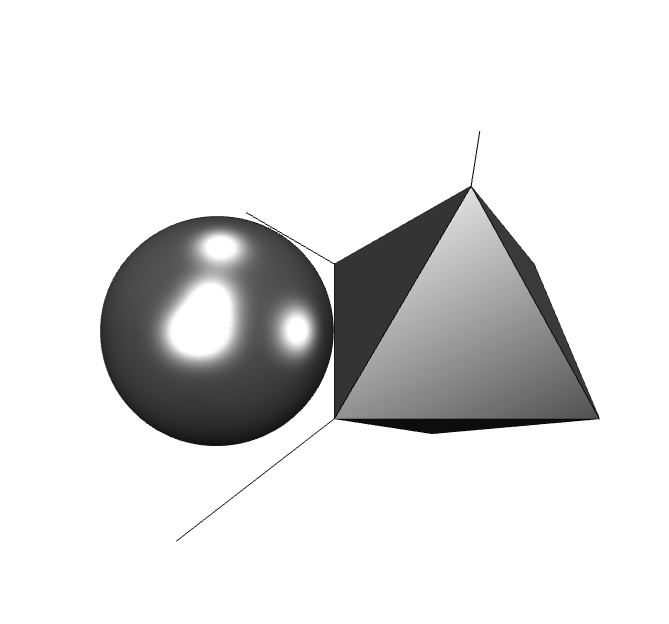} }}
\caption{After preconditioning, the $\ell_2$ loss ellipse turns to a sphere.  This ensures that it intersects the polyhedron in the plane. After preconditioning, the Lasso correctly selects the true model. } \label{fig:geopreconditioned}
\end{figure}

%\begin{figure}[htbp] %  figure placement: here, top, bottom, or page
%   \centering
%\includegraphics[height = 3in, width = 3in]{Simulations/preconditionedDesign.png}
%   \caption{After preconditioning, the $\ell_2$ loss ellipse turns to a sphere.  This ensures that it intersects the diamond in the plane formed by the axes that point to the left. After preconditioning, the Lasso correctly selects the true model. }
%   \label{fig:geopreconditioned}
%\end{figure}

%\subsection{setup}
%Suppose you observe the data $(Y_i, x_i) \in R \times \R^p$ for $i = 1, \dots, n$ from the model
%\[ Y_i  = x_i \truebeta + \epsilon_i,\]
%where several elements of $\truebeta$ are zero and  $\epsilon_1, \dots, \epsilon_n$ are iid with expectation zero and finite variance.  We aim to estimate $\truebeta \in \R^p$ from the data.  Define $X\in \R^{n \times p}$ such that the $i$th row is $X_i$ and define $Y, \epsilon \in \R^n$ such that the $i$th elements are $Y_i$ and $\epsilon_i$ respectively.  
%The Lasso citep{Tibshirani1996} estimator is a method of sparse regression  defined as the solution to a penalized least
%squares problem (with regularization parameter $\lambda$):
%\begin{equation}
%\hbet =\arg\min_{b} \frac{1}{2n}\|\Y-\Xb\|_2^2+\lam
%\|b\|_{1}\label{lasso},
%\end{equation}
%where for some vector $x \in \R^k$, $\|x\|_r = (\sum_{i=1}^k |x_i|^r)^{1/r}$.  
%
%The Lasso has had extensive empirical use and it has generated significant theoretical interest.  Several conditions (restricted isometry property, 

%

In both Figure \ref{fig:geocordesign} and Figure \ref{fig:geopreconditioned}, $\truebeta = (1,1,0)$.  In both figures, the third dimension is represented by the axis that points up and down.  Thus, if the ellipse intersects the constraint set in the (horizontal) plane formed by the first two dimensions, then the Lasso estimates the correct sign.  In Figure \ref{fig:geocordesign}, the design matrix fails the irrepresentable condition and  the elongated shape of the ellipse forces $\hat \beta(c)$ off of the true plane.  This is illustrated in the right panel of Figure \ref{fig:geocordesign}.  High dimensional regression techniques that utilize concave penalties avoid this geometrical misfortune by changing the shape of the constraint set.  Where the $\ell_1$ ball has a flat surface, the non-convex balls flex inward, dodging any unfortunate protrusions of the ellipse.  As preconditioning acts on the opposite set, it restricts the protrusions in the ellipse.  

In Figure \ref{fig:geopreconditioned}, the elongated direction of the ellipse shrinks down, and the ellipse is puffed out into a sphere; it then satisfies the irrepresentable condition, and $\hbet$ lies in the true plane.   Therefore, in this example, preconditioning circumvents the stringent condition for sign consistency.  When $n>p$ preconditioning makes the ellipse a sphere.  When $p>n$,  preconditioning can drastically reduce the pairwise correlations between columns, thus making low dimensional projections of the ellipse more spherical.   Both Figure \ref{fig:geocordesign} and Figure \ref{fig:geopreconditioned} were drawn with the R library rgl.

%In   Because the the design matrix fails the irrepresentable condition,   
%
%
%
%Because $X$ has three columns, the optimization problem \ref{constrainedlasso} is in three dimensions.  Figures \ref{fig:geocordesign} and \ref{fig:geopreconditioned} are both two dimensional projections
%
%
%These figures are two dimensional representations of the 
%
%
%This subsection illustrates the geometric effects of preconditioning for the Lasso.  
%
%, the two solid shapes represent the two 
%
%In the figures, the $\ell_1$ ball  represents the constraint set for the Lasso and the ellipse represents the $\ell_2$ data fidelity term.  For a fixed $\ell_1$ ball, one increase the size of the ellipse until it intersects the $\ell_1$ ball.  The point of intersection represents $\hat \beta(\lambda)$ for some $\lambda$.  
%
%In both figures, $\truebeta$ lies in the ``true plane" that is formed by the two axes that point to the left.  For the Lasso to work, the ellipse should intersect the diamond in this plane. 

Figure \ref{fig:geopreconditioned} illustrates why the Puffer Transformation is so named.  We call $\pre$ the Puffer Transformation in reference to the pufferfish.  In its relaxed state, a pufferfish looks similar to most other fish, oval and oblong.  However, when a pufferfish senses danger, it defends itself by inflating its body into a sphere.  In regression, if the columns of the design matrix are correlated, then the contours of the loss function $\ell_2(b) = \|Y - \X b\|^2$ are oval and oblong as illustrated in Figure \ref{fig:geocordesign}.  
%The shape of these curves reflect the covariance structure of the standard OLS estimator.  
If a data analyst has reason to believe that the design matrix might not satisfy the irrepresentable condition, then she could employ the Puffer Transformation to ``inflate" the smallest singular values, making the contours of $\ell_{2, \pre} (b) = \|\pre Y -\pre \X b\|^2$ spherical.  
Although these contours are not faithful to the covariance structure of the standard OLS estimator, the spherical contours are more suited to the geometry of the Lasso in many settings.

\section{Low dimensional Results} \label{lowdim}
This section demonstrates that for $n \ge p$, after the Puffer Transformation, the Lasso is sign consistent with a minimal assumption on the design matrix $\X$.  
%increasing the correlation between the columns of $\X$ decreases $\sigma_{\min}$.  With the Puffer Transform, this requirement ensures that the error terms $\tilde e = \pre  \e$ have a bounded variance;
When $n\ge p$, the preconditioned design matrix is orthonormal.
\[(\pre  \X)^T \pre \X = VDU^T U D^{-1} U^T U D^{-1} U^T UDV^T = I\] 
The irrepresentable condition makes a requirement on 
\[\Xb' \Xa\left(\Xa'\Xa\right)^{-1}\]
%and the RIP bounds the largest and smallest eigenvalues of 
%\[X(T)'X(T)\]
%for all subsets $T \subset \{1, \dots, p\}$ of a bounded size. 
Since the Puffer Transformation makes the columns of the design matrix orthonormal, $(\pre \Xb) \pre \Xa = 0$,  %and $(\pre X(T))' \pre X(T) = I$, 
satisfying  irrepresentable condition.

\begin{theorem}
\label{thm:fixeddim}
Suppose that data $(\X,\Y)$ follows the linear model described in Equation
(\ref{regeq}) with  iid Gaussian noise $\e \sim N(0, \sigma^2I_n)$. Define the singular value decomposition of $X$ as $\X = UDV'$.  Suppose that $n\geq p$ and $\X$ has rank $p$. We further assume that $\Lambda_{\min}(\frac{1}{n}X'X) \geq \tilde C_{\min}>0$. Define the \textbf{Puffer Transformation	},
$\pre  = UD^{-1}U^T.$ Let $\tilde \X = F \X$ and $\tilde Y = F Y$. Define

\[\tilde \beta(\lambda) = \arg\min_b \frac{1}{2} \|\tilde Y - \tilde \X b\|_2^2 + \lambda  \| b\|_1.\]

If $\min_{j\in S}|\truebeta_j| \geq 2  \lambda$, 
then with probability greater than
$$1-2p \exp\left\{-\frac{n\lam^2\tilde C_{\min}}{2\sigma^2}\right\} $$
 $\tilde \beta(\lambda) =_s \truebeta$.
\end{theorem}

\textbf{Remarks.} Suppose that $\tilde C_{\min}>0$ is a constant. If  $p, \min_{j\in S}|\truebeta_j|$ and $\sigma^2$ do not change with $n$, then choosing $\lambda$ such that (1) $\lambda  \rightarrow 0$ and (2) $\lambda^2 n\rightarrow \infty$ ensures that $\tilde \beta (\lambda)$ is sign consistent.  One possible choice is $\lambda = \sqrt{\frac{\log n}{n}}$.

 From classical linear regression, we know that increasing the correlation between columns of $\X$ increases the variance of the standard OLS estimator;  correlated predictors make estimation more difficult.  This intuition translates to preconditioning with the Lasso; increasing the correlation between the columns of $\X$ decreases the smallest singular value of $\X$,   increasing the variance of the noise terms $cov(\pre  \e) = U D^{-2} U'$, and  weakening the bound in Theorem \ref{thm:fixeddim}.  When the columns of $\X$ are correlated, then $\tilde C_{\min}$ is small and Theorem \ref{thm:fixeddim} gives a smaller probability of sign consistency.

Theorem \ref{thm:fixeddim} applies to many more sparse regression methods that use penalized least squares.  After preconditioning, the design matrix is orthogonal and several convenient facts follow.  First, if the penalty decomposes,
\[pen(b,\lam) = \sum_{j=1}^p pen_j(b_j, \lam)\]
so that $pen_j$ does not rely on $b_k$ for $k \ne j$, then the penalized least squares method admits a closed form solution.  If it is also true that all the $pen_j$'s are equivalent and have a cusp at zero (like the Lasso penalty), then the method selects the same sequence of models as the Lasso and correlation screening (i.e. select $X_j$ if $|cor(Y,X_j)| \ge \lambda$) \citep{fan2008sure}.  For example, both SCAD and MCP satisfy these conditions.  If  a method selects the same models as the Lasso, and the Lasso is sign consistent, then this method is also sign consistent. Thus, Theorem \ref{thm:fixeddim} implies that (1) preconditioned correlation screening, (2) preconditioned SCAD, and (3)  preconditioned  MCP are sign consistent under some mild conditions (similar to the conditions in Theorem \ref{thm:fixeddim}).  However, in high dimensions, $\pre X$ is no longer orthogonal.  So, the various methods could potentially estimate different models. 

\section{High dimensional Results}\label{highdim}

The results for $p>n$ are not as straightforward because the columns of the design matrix cannot be orthogonal.  However, the results in this section suggest that for many high dimensional design matrices $\X$, the matrix $\pre \X$ satisfies the stringent assumptions of the Lasso.  In fact, the simulation results in the following section suggest that preconditioning offers dramatic improvements in high dimensions.

Before introducing Theorem \ref{unifTheorem}, Figure \ref{correduce} presents an illustrative numerical simulation to prime our intuition on preconditioning in high dimensions.   In this simulation, $n = 200, p = 10,000$, and each row of $\X$ is an independent Gaussian vector with mean zero and covariance matrix $\Sigma$.  The diagonal of $\Sigma$ is all ones and the off diagonal elements of $\Sigma$ are all $.9$.   The histogram in Figure \ref{correduce} includes both the distribution of pairwise correlations between the columns of $\X$ and the distribution of pairwise correlations between the columns of $\pre \X$.   Before the transformation, the pairwise correlations between the columns of $\X$ have an average of $.90$ with a standard deviation of $.01$.  After the transformation, the average correlation is $.005$, and the standard deviation is $.07$.   Figure \ref{correduce} shows this massive reduction in correlation. The histogram has two modes; the left mode corresponds to the distribution of pairwise correlations in $\pre \X$, and the right mode corresponds to the distribution of correlations in $\X$.

\begin{figure}[htbp]
\begin{center}
\includegraphics[width = 5in]{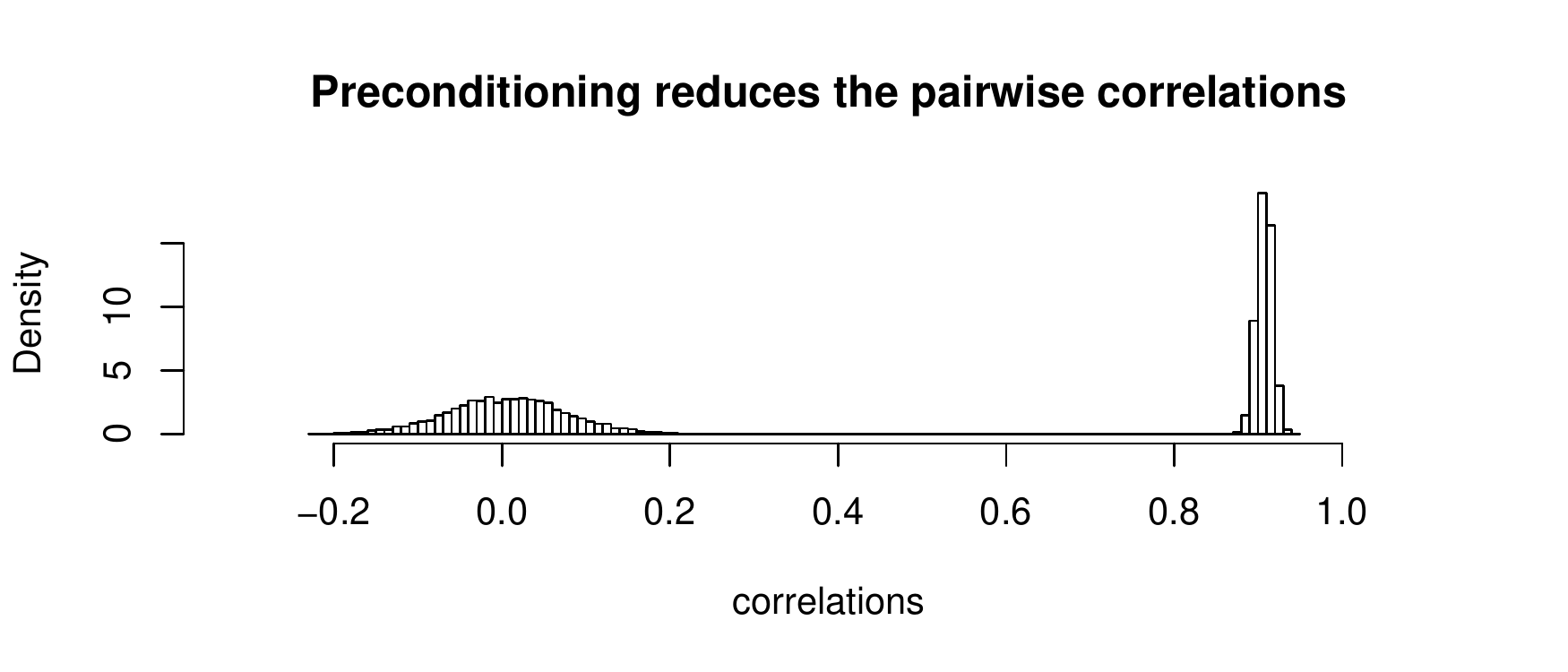}
\caption{This figure displays the pairwise correlations between the columns of $\pre \X$ (on the left) and the pairwise correlations between the columns of $\X$ (on the far right).  In this simulation, the rows of $\X$ are iid multivariate Gaussians with mean zero and covariance matrix $\Sigma$,  $\Sigma_{ii} = 1$ and $\Sigma_{i \ne j} = .9$.  The figure was created by first sampling 10,000 $(i,j)$ pairs without replacement such that $i \ne j$ and then computing $\mbox{cor}(\pre X_i, \pre X_j)$ and $\mbox{cor}(X_i, X_j)$ for each of these 10,000 pairs. Before preconditioning, the pairwise correlations are much larger.  In this setting, preconditioning reduces the pairwise correlations.
}
\label{correduce}
\end{center}
\end{figure}

By reducing the pairwise correlations, preconditioning helps the design matrix satisfy the irrepresentable condition.  For example, if the first twenty elements of $\truebeta$ are positive and the other 9,980 elements equal zero, 
%
%So, the columns of $\pre \X$ are much less correlated with each other. 
% The top plot in Figure \ref{corboost} displays the correlation between $Y$ and the columns of $\X$.  The ten dots correspond to the ten columns of $\X$ with nonzero $\truebetai$.  The bottom plot in Figure \ref{corboost} displays the same information after the transformation.  After the transformation, all of the variable become much less correlated with $Y$, including the relevant variables. However, the transformation makes the correlations of the relevant predictors more extreme relative to the other variables' correlations.  Under these simulation conditions, the Puffer Transformation makes the marginal correlations between the design matrix and the outcome more informative.  Additionally, 
%the transformation can improve the performance of sparse penalized regression techniques such as Lasso, SCAD, and MCP, helping them to discover the relevant variables more easily.  
then the Puffer Transformation makes the design matrix satisfy the irrepresentable condition in this simulation.  Recall that the irrepresentable condition bounds the quantity 
\begin{equation} \label{IC2}
\left\|\Xb^T\Xa\left(\Xa^T\Xa\right)^{-1}\overrightarrow{b}\right\|_{\infty} < 1 - \eta
\end{equation}
for $\eta > 0$.  From the simulated data displayed in Figure \ref{correduce}, the left hand side of Equation (\ref{IC2}) evaluated with the matrix $\X$ equals $1.09$; evaluated with $\pre \X$, it equals $.84$.   In this example, the Puffer Transformation circumvents the irrepresentable condition.

\subsection{Uniform distribution on the Stiefel manifold}

When $n\ge p$, the columns of $\pre \X$ are orthogonal.  When $p \ge n$, the \textit{rows} are orthogonal.  $\pre \X$ lies in the Stiefel manifold,
\[\pre \X \in V(n,p) = \{ V \in \R^{n \times p} : VV' = I_n\}.\]
Further, under any unitarily invariant norm,  $\pre \X$ is the projection of $\X$ onto $V(n,p)$  \citep{fan1955some}.  Since $V(n,p)$ is a bounded set, it has a uniform distribution.
\begin{definition}[\cite{chikuse2003statistics}]
A random matrix $V$ is uniformly distributed on $V(n,p)$, written $V \sim uniform(V(n,p))$, if  the distribution of $V$ is equal to the distribution of $V O$ for any fixed $O$ in the orthogonal group of matrices $O(p,\R)$, where
\[O(p,\R) = \{O \in \R^{p \times p} : OO' = I_p\}.\] 
\end{definition}
In this section, Theorem \ref{unifTheorem} shows that if $V$ is uniformly distributed on $V(n,p)$, then after normalizing the columns of $V$ to have equal length, the matrix satisfies the irrepresentable condition with high probability in certain regimes of $n, p$, and $s$.  Propositions \ref{unifGaussian} and \ref{unifCor} give two examples of random design matrices $\X$ where $\pre \X$ is uniformly distributed on $V(n,p)$.

\begin{theorem} \label{unifTheorem}
Suppose $X \in \R^{n\times p}$ is distributed  $uniform(V(n, p))$ and $n > c s^4$, where $s$ is the number of relevant predictors and $c$ is some constant, then asymptotically, the irrepresentable condition holds for normalized version of $X$  with probability no less than
\[1 - 4p^2 \exp(-n^{1/2}),\]
where the normalized version of $X$, denoted $\tilde X$ is defined as
$\tilde X_{ij} = X_{ij} / \sqrt{\sum_{i=1}^n X_{ij}^2}.$
%XXX Jinzhu, is this an asymptotic bound or does this bound work for a fixed $n, p, s$?XXX
\end{theorem}

The proof of Theorem \ref{unifTheorem} relies on the fact that if $|\mbox{cor}(X_i, X_j)| \le c/(2s - 1)$ for all pairs $(i,j)$, then $\X$ satisfies the irrepresentable condition.  This line of argument requires that $n > c s^4$.  A similar argument applied to a design matrix $\X$ with iid $N(0,1)$ entries only requires that $n > c s^2$ (this is included in Theorem \ref{thm:ICG} in the appendix Section \ref{subsec:Smanifold}).  If both of these are tight bounds, then it suggests that the irrepresentable condition for $V \sim uniform(V(n,p))$ is potentially more sensitive to the size of $s$ than a design matrix with iid $N(0,1)$ entries.  The final simulation in Section \ref{sim} suggests that this difference is potentially an artifact of our proof.  Both distributions are almost equally sensitive to $s$.  If anything, $uniform(V(n,p))$ has a slight advantage in our simulation settings. 

%This dependence on $s$ will be explored in the simulations in Section \ref{sim}. 

Propositions \ref{unifGaussian} and \ref{unifCor} give two models for $\X$ that makes  $\pre \X $ uniformly distributed  on the Stiefel manifold.  

\begin{prop} \label{unifGaussian}
If the elements of $\X$ are independent $N(0,1)$ random variables, then $\pre \X \sim uniform(V(n,p))$.
\end{prop}

\begin{proof}
Define $U, V, D$ by the SVD of $\X$, $\X = UDV'$.  For a fixed $O \in O(p,\R)$, $\tilde V' = V' O$ is an element of $V(n,p)$.  Therefore, the SVD of $\X O = UD\tilde V'$. This implies that both $\X$ and $\X O$ have the same Puffer Transformation $\pre = UD^{-1}U'$.  This yields the result when combined with the fact that $\X \stackrel{d}{=} \X O$, 
\[\pre \X \stackrel{d}{=} \pre \X O. \]
\end{proof}

\begin{prop} \label{unifCor}
Suppose that $U_\Sigma \in \R^{p \times p}$ is drawn uniformly from the orthonormal group of matrices $O(p,\R)$
and $D_\Sigma \in \R^{p \times p}$ is a diagonal matrix with positive entries (either fixed or random).  Define $\Sigma = U_\Sigma D_\Sigma U_\Sigma^T$ %as the eigendecomposition of $\Sigma$ 
and suppose the rows of $\X$ are drawn independently from $N(0, \Sigma)$,  then $\pre \X  \sim uniform(V(n,p))$. 
\end{prop}

\begin{proof}

%  This means that if $O \in O(p,R)$, then $U_\Sigma \stackrel{d}{=} O U_\Sigma$.  
%   Suppose the rows of $\X$ are iid $N(0, \Sigma)$ where the covariance matrix is random in the following manner.  $U_\Sigma$ is drawn uniformly from the orthonomal group $O(p,R)$ and $\Sigma = U_\Sigma D_\Sigma U_\Sigma^T.$  Assume that $D_\Sigma$ is fixed with $\lambda_i$ down the diagonal and $trace(D_\Sigma) = p$.  
   
%To see why $\pre \X$ in Example \ref{unifCor} is uniformly distributed on $V(n,p)$, n
Notice that if
 $Z \in \R^{n \times p}$ has iid $N(0, 1)$ elements and $Z$ is independent of $U_\Sigma$, then  $\X \stackrel{d}{=} Z \Sigma^{1/2}$.  
%Similar to before, the Puffer Transformation is equal to $(\X \X^T)^{1/2}$.  Define the new design matrix as $V = (\X \X^T)^{1/2} \X $.  
The following equalities in distribution follow from the fact that for any $O \in O(p,\R)$, $U_\Sigma \stackrel{d}{=} OU_\Sigma$ and $ZO \stackrel{d}{=} Z$. 
\begin{eqnarray*}
\pre \X &=& ( \X \X^T)^{-1/2} \X\\
&\stackrel{d}{=}& (Z\Sigma Z^T)^{-1/2}Z \Sigma^{1/2}\\
&\stackrel{d}{=}& (ZOU_\Sigma D_\Sigma U_\Sigma^TO^TZ^T)^{-1/2}ZOU_\Sigma D_\Sigma^{1/2} U_\Sigma^TO^T\\
&\stackrel{d}{=}& (ZU_\Sigma D_\Sigma U_\Sigma^TZ^T)^{-1/2}ZU_\Sigma D_\Sigma^{1/2} U_\Sigma^TO^T\\
&=& \pre \X O
\end{eqnarray*}
Thus, $\pre \X \sim uniform(V(n,p))$.
%; $V \stackrel{d}{=} V O$ for any $O \in O(p,R)$. 
\end{proof}

Proposition \ref{unifCor} presents a scenario in which  the distribution of $\pre \X$  is independent of the eigenvalues of $\Sigma$.  Of course, if $\Sigma$ has a large condition number, then the transformation $\pre = (\X \X^T)^{-1/2}$ could potentially induce excessive variance in the noise.

In practice, one usually gets a single design matrix $\X$ and a single preconditioned matrix $\pre \X \in V(n,p)$.  It might be difficult to argue that $\pre \X \sim uniform(V(n,p)).$  Instead, one should interpret Theorem \ref{unifTheorem} as saying that nearly all matrices in $V(n,p)$ satisfy the irrepresentable condition. 

So far, this section has illustrated how preconditioning circumvents the irrepresentable condition in high dimensions.  This next theorem assumes $\pre \X$ satisfies the irrepresentable condition and studies when the preconditioned Lasso estimator is sign consistent,  highlighting when $\pre$ might induce excessive noise.  

\begin{theorem}\label{highdimthm}
Suppose that data $(\X,\Y)$ follows the linear model described in Equation (\ref{regeq}) with  iid normal noises $\e \sim N(0, \sigma^2I_n)$. Define the singular value decomposition of $X$ as $\X = UDV^T$.  Suppose that $p\geq n$. We further assume that $\Lambda_{\min}(\frac{1}{n}\Xa^T\Xa) \geq \tilde C_{\min}$  and ${\min_{i}}(D^2_{ii}) \geq  p d_{\min}$ with  constants $\tilde C_{\min}>0$ and $d_{\min}>0$. For $\pre  = UD^{-1}U^T,$ define $\tilde Y = F Y$ and $\tilde \X = F \X$. Define

\[\tilde \beta(\lambda) = \arg\min_b \frac{1}{2} \|\tilde Y - \tilde \X b\|_2^2 + \lambda  \| b\|_1.\]

Under the following three conditions, 
\begin{enumerate}
\item $\left\|\tilde \X(S^c)^T\tilde\X(S)\left(\tilde\X(S)^T\tilde\X(S)\right)^{-1}\overrightarrow{b}\right\|_{\infty}\leq 1-\eta, $
\item $\Lambda_{\min}\left(\tilde \X(S)^T\tilde \X(S)\right) \geq \frac{n}{cp}$
\item $\min_{j\in S}|\truebetaj| \geq 2\lam \sqrt{scp/n} $ 
\end{enumerate}
where $c$ is some constant; we have 
\begin{equation} \label{probboundThm}
P\left(\tilde \beta(\lambda) =_s \truebeta\right) > 1-2\p \exp\left\{-\frac{p\lam^2\eta^2d_{\min}}{2\sigma^2 }\right\} .
\end{equation}
\end{theorem}

This theorem explicitly states when preconditioning will lead to sign consistency.  %In that case, $1/\|T\|^2$ would replace the $d_{\min}$ in the exponent of the probability bound of Equation \ref{probbound} ($\|T\|$ is the spectral norm of $T$ with an appropriate scale). {\color{red} [I do not quite agree with this statement. First, I do not think $1/\|F\|^2 = p d_{\min}$. Second condition (2) heavily relies on our Puffer transformation. I've given some comments just after Theorem 3.]}}  
Note that $d_{\min}$ in the exponent of Equation (\ref{probboundThm}) corresponds to the amount of additional noise induced by the preconditioning.   For $p\gg n$, the assumption $\min_i(D_{ii}^2) \geq p d_{min}$ is often satisfied.  For example, it holds if $\X\sim N(0, \Sigma)$  and the eigenvalues of $\Sigma$ are lower bounded.   To see this, define $\X = Z \Sigma^{1/2}$ for $Z\in\R^{n\times p}$ with $Z_{ij} \sim_{i.i.d.} N(0,1)$.  Then,
\[D_{ii}^2 \ge  \min_{w\in \R^n, \|w\|_2 = 1} \|w' \X\|_2^2 = \min_{w\in \R^n, \|w\|_2 = 1} \|w' Z \Sigma^{1/2}\|_2^2  \ge \min_{w\in \R^n, \|w\|_2 = 1} \|w' Z\|_2^2 \lambda_{\min}(\Sigma),\]
where $\lambda_{\min}(\Sigma)$ is the smallest eigenvalue of $\Sigma$.  With high probability, $\frac{1}{p}\|w' Z\|_2^2$ is bounded below  \citep{DavidsonS2001}.   Thus, $\min_i(D_{ii}^2) \geq p d_{min}$ holds for some constant $d_{\min}$ with high probability.

The enumerated conditions in Theorem \ref{highdimthm} correspond to standard assumptions for the Lasso to be sign consistent.  Condition (1) is the irrepresentable condition applied to $\tilde \X$.  Condition (2) ensures first, that the columns in $\tilde \X(S)$ are not too correlated and second, that the columns in $\tilde \X(S)$ do not become too short since $\pre$ rescales the lengths of the columns.  From Section \ref{subsec:Smanifold} and the discussion of the uniform distribution on the Stiefel manifold, most matrices in $V(n,p)$ satisfy  condition (1) as long as $s = o(n^{1/4})$ and $p^2 = o(\exp(\sqrt{n}))$.  Similarly, most matrices satisfy condition (2); Theorem \ref{thm:stiefel} in Appendix Section \ref{subsec:Smanifold} states that the diagonals of $\tilde \X(S)^T\tilde \X(S)$  concentrate around $n/p$ and the off-diagonals concentrate around $0$.     Condition (3) in Theorem \ref{highdimthm} ensures that the signal strength does not decay too fast relative to $\lambda$.  The next corollary chooses a sequence of $\lambda$'s.
\begin{corollary} 
Under the conditions of Theorem \ref{highdimthm},  if $\min_{j\in S}|\truebeta_j|$ is a constant and $\\ n/(s\log p) \rightarrow \infty$, then setting $\lambda^2 = \sqrt{n\log( p )/(sp^2)}$ ensures sign consistency.
\end{corollary}
%
% , there exists some $\lambda$ such that Condition (3) holds and the probability that  $\tilde \beta(\lambda) =_s \truebeta$ goes to 1, if we further have (1) $\lambda  \sqrt{sp/n} \rightarrow 0$ and (2) $\lambda^2 p / \log p \rightarrow \infty$. A possible choice of $\lambda$ is such that $\lambda^2 = \sqrt{n\log( p )/(sp^2)}$, when $n/(s\log p) \rightarrow \infty.$

  Theorem \ref{highdimthm} highlights the tradeoff between (a) satisfying the irrepresentable condition and (b) limiting the amount of additional noise created by preconditioning.  Interestingly, even if $\X$ satisfies the irrepresentable condition, the Puffer Transformation might still improve sign estimation performance as long as the preconditioner can balance the tradeoff.  To illustrate this, presume that both $\tilde \X$ and $\X$ satisfy the irrepresentable condition with constants $\tilde \eta$ and $\eta$ respectively.  Preconditioning improves the bound in Equation \ref{probboundThm} if $\tilde \eta^2 d_{\min} > \eta^2$. Alternatively, if $d_{\min}$ is very small,  then the Puffer Transformation will induce excess noise and $\tilde \eta^2 d_{\min} < \eta^2$, weakening the upper bound.

\section{Simulations} \label{sim}

%To understand the effect of preconditioning, the simulations in this section (a) (b) compare the model selection performance to alternative methods...

The first two simulations in this section compare the model selection performance of the preconditioned Lasso to the standard Lasso, Elastic Net, SCAD, and MC+.  The third simulation compares the uniform distribution on $V(n,p)$ to the distribution that places iid Gaussian ensemble distribution, estimating how often these distributions satisfy the irrepresentable condition.

\subsection{Choosing $\lam$ with BIC} \label{olsbic}

%In a brief set of simulation studies (not shown) it appeared that, even in high dimensions, the Lasso, SCAD, and MC+ select similar models when the data is preconditioned.\footnote{Section \ref{lowdim} explains why these methods with preconditioning select the same model when $n\ge p$.}  For this reason, only the Lasso is preconditioned in the following simulations. 

After preconditioning, the noise vector $\pre \e$ contains statistically dependent terms that are no longer exchangeable.  This creates problems for selecting the tuning parameter in the preconditioned Lasso.  For example, if one wants to use cross-validation, then the test set should not be preconditioned.  After several attempts, we were unable to find a cross-validation procedure that led to adequate model selection performance.  In this section, there are two sets of simulations that correspond to two ways of choosing $\lam$.    The first set of simulations choose $\lam$ using a BIC procedure.  To ensure that the results are not due to the BIC procedure, the second set of simulations choose the $\lam$ such that it selects the first model with ten degrees of freedom.

The first set of simulations choose $\lam$ with the following procedure.

\vspace{.1in}

\noindent
\fbox{
\parbox{5.7in}{
\vspace{.1in}
\noindent \textbf{OLS-BIC; to choose a model in a path of models.}

\noindent 
1) For df = 1:40 \begin{enumerate}
	\item[a)] starting from the null model, select the first model along the solution path with df degrees of freedom.
	\item[b)] use the df model to fit an OLS model.  The OLS model is fit with the original (un-preconditioned) data. 
	\item[c)] compute the BIC for the OLS model. 
\end{enumerate}
\hspace{.15in} end

\vspace{.1in}

2) Select the tuning parameter that corresponds to the model with the lowest OLS-BIC score.
%\vspace{.1in}
}}
\vspace{.1in}

The OLS models were fit with the R function \texttt{lm} and the BIC was computed with the R function \texttt{BIC}.

In this simulation and that of Section \ref{top10},  $n = 250$, $s= 20$, and $p$ grows along the horizontal axis of the figures (from 32 to $32\,768$).  All nonzero elements in $\truebeta$ equal ten and $\sigma^2 = 1$.  The rows of $\X$ are mean zero Gaussian vectors with constant correlation $\rho$.

\begin{figure}
\centering
\mbox{\subfigure{\includegraphics[width=1.7in]{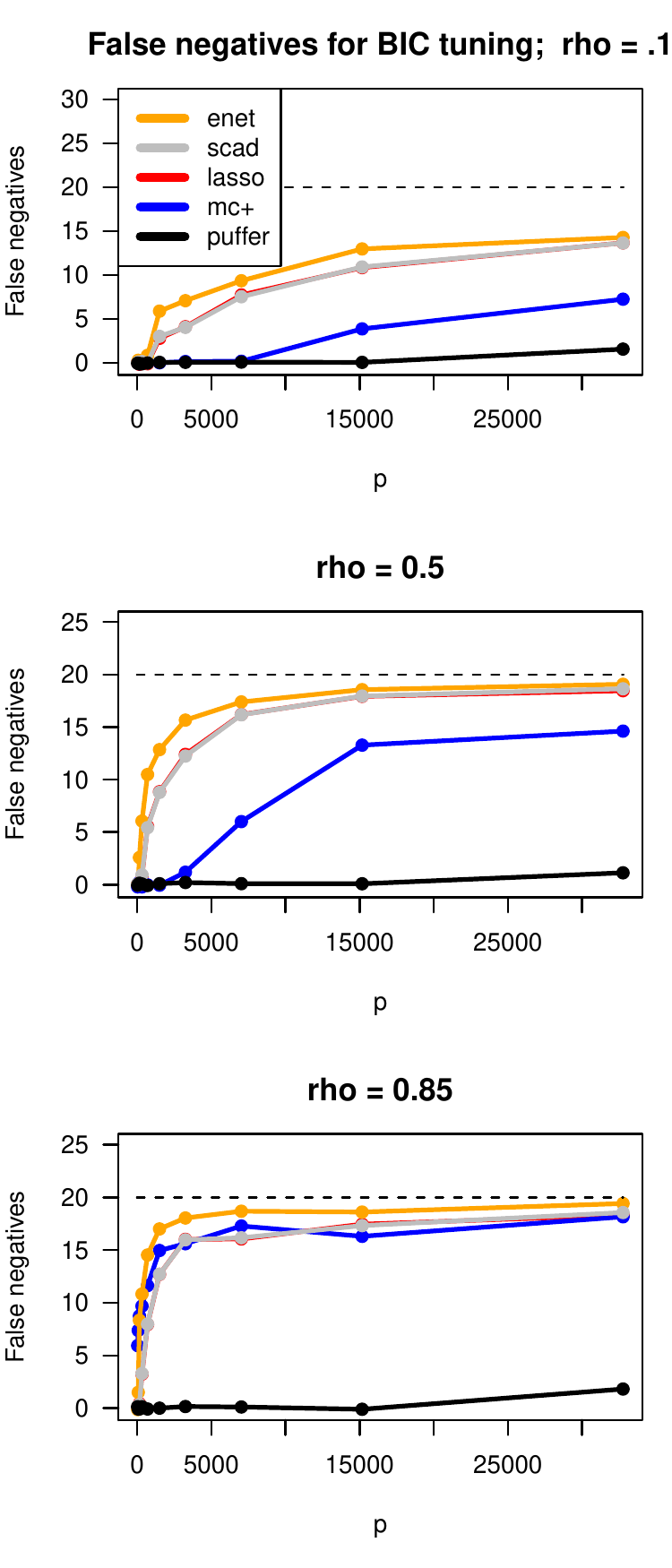}}\quad
\subfigure{\includegraphics[width=1.7in]{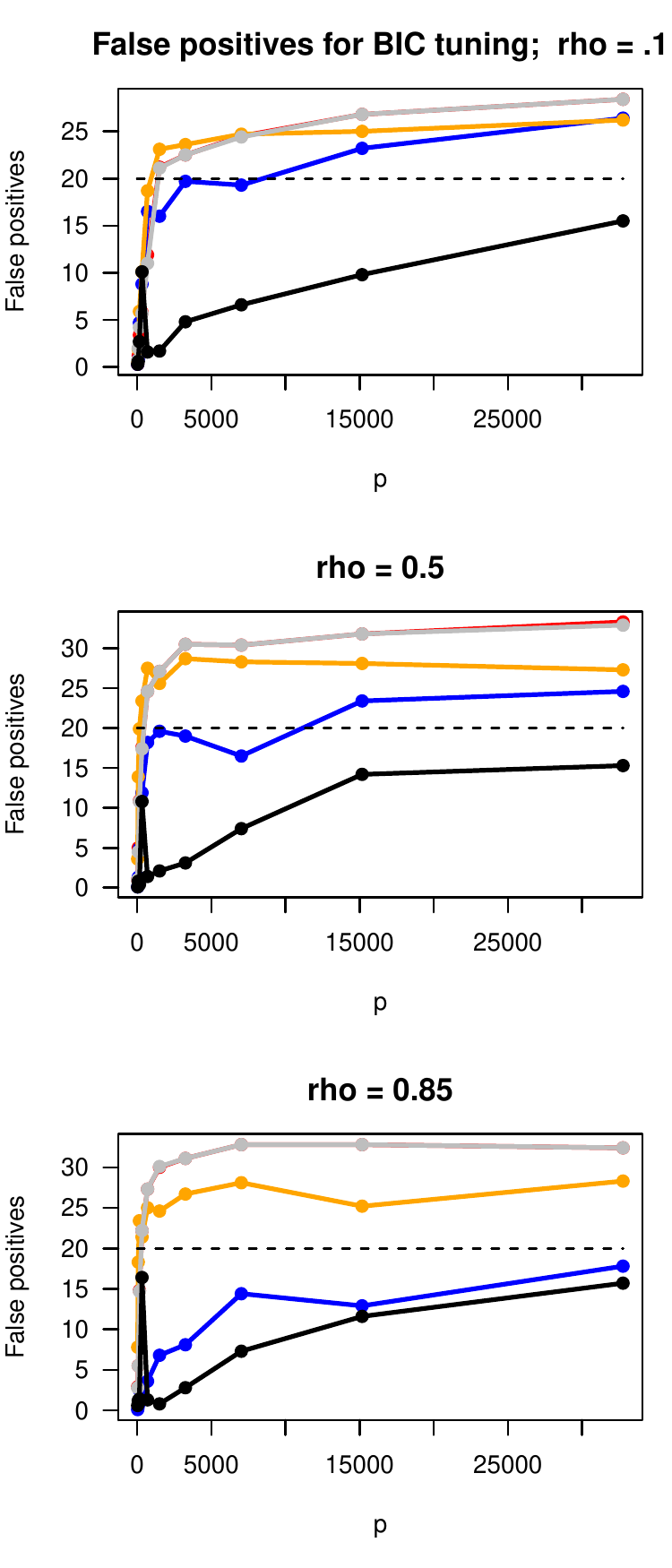} }\quad
\subfigure{\includegraphics[width=1.7in]{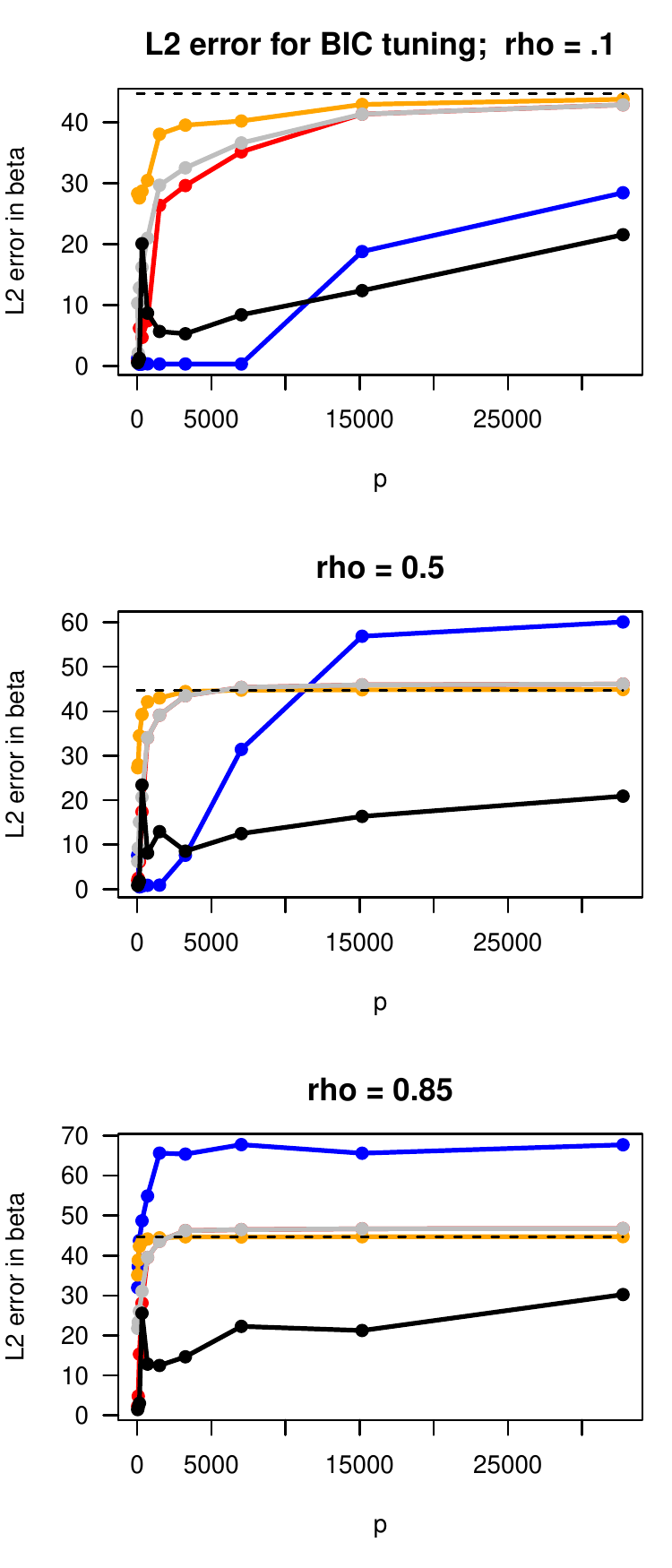} }}
\caption{In all figures, $n = 250$ and $p$ grows along the horizontal axis.    Each row of $\X$ is an independent multivariate Gaussian with constant correlation $\rho$ in the off diagonal of the correlation matrix. The figure above has three rows and three columns of plots.  The three rows correspond to different values of $\rho$.  From top to bottom, $\rho$ is $.1, .5,$ and $.85$.  The three columns correspond to the number of false negatives (on left), the number of false positives (in middle), and $\|\hbet - \truebeta\|_2$ (on right).  The tuning parameter for each method is selected with the OLS-BIC procedure described in Section \ref{olsbic}.  The results from the preconditioned Lasso appear as a solid black line.  Note that the number of false negatives cannot exceed $s = 20$.  In the plots on the left side, a dashed horizontal line at 20 represents this limit. For scale, this dashed line is also included in the false positive plots.  In the $\ell_2$ error plots, the dashed line corresponds to the $\ell_2$ error for the estimate $\hat \beta = 0$.   For both $\rho = .5$ and $.85$, the competing methods miss a significant fraction of the true nonzero coefficients and only at the very end, for $p > 32,000$, does the preconditioned lasso start to miss any of the true coefficients.  At the same time, the preconditioned Lasso accepts fewer false positives than the alternative methods. 
} \label{fig:BIC}
\end{figure}

In Figure \ref{fig:BIC}, there are three columns of plots and three rows of plots.  The rows correspond to different levels of correlation $\rho$;  $.1$ in the top row, $.5$ in the middle row, and $.85$ on the bottom row.  The columns correspond to different measurements of the estimation error;  the number of false negatives on the left, the number of false positives in the middle, and the $\ell_2$ error $\|\hbet - \truebeta\|_2$ on the right.   Each data point in every figure comes from an average of ten simulation runs.  In each run, both $\X$ and $Y$ are resampled.

%Figure \ref{fig:BIC} displays the number of false negatives (on left), the number of false positives (in the middle), and the $\ell_2$ error $\|\hbet - \truebeta\|_2$ (on the right).%\footnote{Note that the $\ell_2$ error is not computed with the OLS fit of $\beta$.}  
%In each of the nine plots, the horizontal axis represents $p$.  The correlation between the columns of $\X$ increases going down the three rows of plots;  $\rho = .1$ in the top three plots, $\rho = .5$ in the middle three plots, and $\rho = .85$ in the bottom three plots.  
In many settings, across both $p$ and $\rho$,  the preconditioned Lasso simultaneously admits fewer false positives and fewer false negatives than the competing methods.  The number of false negatives when $\rho = .85$ (displayed in the bottom left plot of Figure \ref{fig:BIC}) gives the starkest example.  However, the results are not uniformly better.  For example, when $p \approx n= 250$, the preconditioned Lasso performs poorly;  this behavior around $p \approx n$ appeared in several other simulations that are not reported in this paper.  Further, %MC+ has smaller $\ell_2$ error when $p$ is small or moderate and the correlation is small or moderate.
  when the design matrix has fewer predictors $p$ and  lower correlations $\rho$, MC+ has smaller $\ell_2$ error than the  preconditioned Lasso.  However, as the correlation increases or the number of predictors grows, the preconditioned Lasso appears to outperform MC+ in $\ell_2$ error.

\subsection{Ensuring that OLS-BIC gives a fair comparison} \label{top10}
%The second set of simulations chooses the first model along the solution path that contains at least ten predictors.  Ten is half the number of nonzero elements in $\truebeta$.\footnote{We initially attempted to stop after 20 predictors.  However, MC+ often failed to select a model this large.  To ensure a fair comparison, we reduced the stopping criterion to 10.  Even at this level, when $\rho = .85$, MC+ failed to select a model with at least 10 predictors in roughly 15\% of the simulations. These simulations were discarded and their results are not represented in any figure.}   While this might be a strange way to choose $\lam$ in practice,  it is useful in this simulation study to ensure that the OLS-BIC method of model selection gives a fair comparison of the various model selection procedures. 

To ensure that the results in Figure \ref{fig:BIC} are not an artifact of the OLS-BIC tuning,  Figure \ref{fig:pgrowing} tunes procedure by taking the first model to contain at least ten predictors.  The horizontal axes displays $p$ on the log scale. The vertical axes report the number of false positives divided by ten.   Each dot in the figure represents an average of ten runs of the simulation, where each run resamples $\X$ and $Y$.

Figure \ref{fig:pgrowing} shows that as $p$ grows large, the standard Lasso, Elastic Net (enet), SCAD, and MC+ perform in much the same way, whereas the preconditioned Lasso drastically outperforms all of them.  For the largest $p$ in the middle and lower graphs, the preconditioned Lasso yields two or three false positives out of ten predictors;  all of the other methods have at least eight false positives out of ten predictors.  These results suggest that the results in Figure \ref{fig:BIC} were not an artifact of the OLS-BIC tuning.

The preconditioned Lasso does not perform well when $p$ is close to $n=250$.  This is potentially due to the instability of the spectrum of $\pre$ when $p$ and $n$ are comparable.  Results in random matrix theory suggest that in this regime, $C_{\min}$ (the smallest nonzero singular value of $\X$) can be very close to zero \citep{silverstein1985smallest}.  When the spectral norm of the Puffer Transform ($1/C_{\min}$) is large, for example when $p \approx n$, preconditioning can greatly amplify the noise, leading to poor estimation performance. As previously mentioned, it is an area for future research to investigate if a Tikhonov regularization of $\pre$ might improve the performance in this regime. 

\begin{SCfigure}[1.1]
  \centering
  \includegraphics[width=2in]    {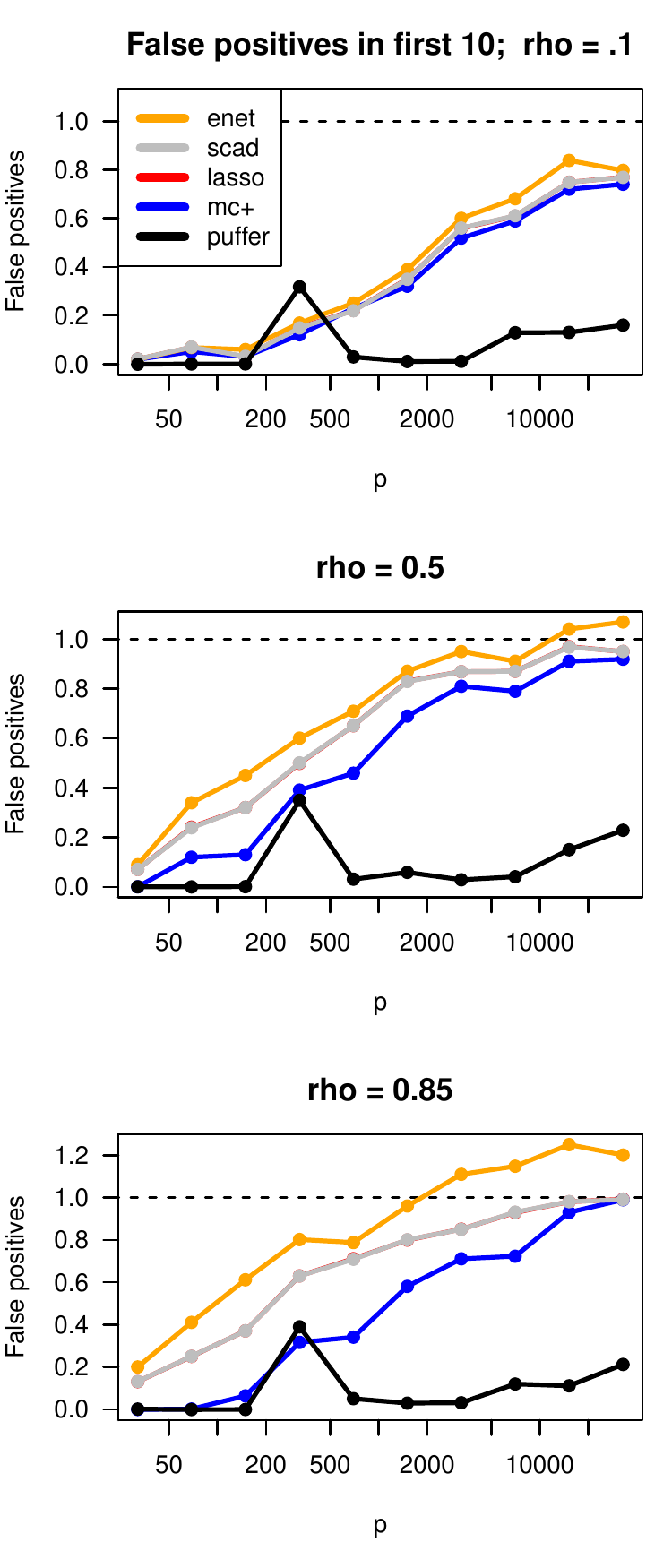}% picture filename
  \caption[width = 2in]{ To ensure that the results in Figure \ref{fig:BIC} are not due to the choice of tuning parameter, this figure chooses the tuning parameter in a much simpler fashion.   Each tuning parameter corresponds to the first estimated model that contains at least ten predictors.   Along the horizontal axis, $p$ increases on the log scale.  The vertical axis displays  the number of false positives divided by ten.  The line for the elastic net can exceed one because the default settings of glmnet sometimes fail to select a model with exactly 10 predictors.  So, the tuning method selects the next biggest model.  Both $\truebeta,$ and $\sigma^2$ are unchanged from the simulations displayed in Figure \ref{fig:BIC} and $\X$ comes from the same distribution.  %  The design matrix has independent rows and constant correlation between columns.  The correlation varies in each plot, from $.1$, to $.5$, to $.85$.  $n = 250$ and there are twenty nonzero elements in $\truebeta$.  Each nonzero element is equal to ten.  The standard deviation of the noise is one.  
  The results in this figure are comparable to the results in Figure \ref{fig:BIC}.  This suggests that the results from OLS-BIC provide a fair comparison of the methods.  \vspace{1in}
  %   and plots $p$ on the log scale.   The results in this figure are comparable to the results in Figure \ref{fig:BIC}, suggesting that the results from OLS-BIC provide a fair comparison. 
  }
   \label{fig:pgrowing}
\end{SCfigure}

All simulations in this section were deployed in R with the packages lars (for the Lasso), plus (for SCAD and MC$+$), and glmnet (for the elastic net).  All packages were run with their default settings.

\subsection{Satisfying the irrepresentable condition on $V(n,p)$} 
Theorem \ref{unifTheorem} shows that, under certain conditions, if $V \sim  uniform(V(n,p))$ then $V$ satisfies the irrepresentable condition with high probability.  One of the more undesirable  conditions is that it requires $n >  cs^4$, where $c$ is a constant and $s$ is the number of relevant predictors.  As is discussed after Theorem \ref{unifTheorem}, if the design matrix contains iid $N(0,1)$ entries, then it is only required that $n>cs^2$.  The simulation displayed in Figure \ref{fig:ircforvnp} compares the irrepresentable condition scores:
\begin{equation} \label{ircscore}
\| IC(X, S) = \X(S^c)' \X(S) \left(\X(S)' \X(S)\right)^{-1} \textbf{1}_{s}\|_\infty,
\end{equation}
where $\textbf{1}_{s} \in \R^{s}$ is a vector of ones, for $\X$ from the iid $N(0,1)$ distribution and for $\X \sim uniform(V(n,p))$.  In these simulations, $n=200$ and $p=10,000$.

\begin{figure}[htbp] %  figure placement: here, top, bottom, or page
   \centering
   \includegraphics[width=5.5in]{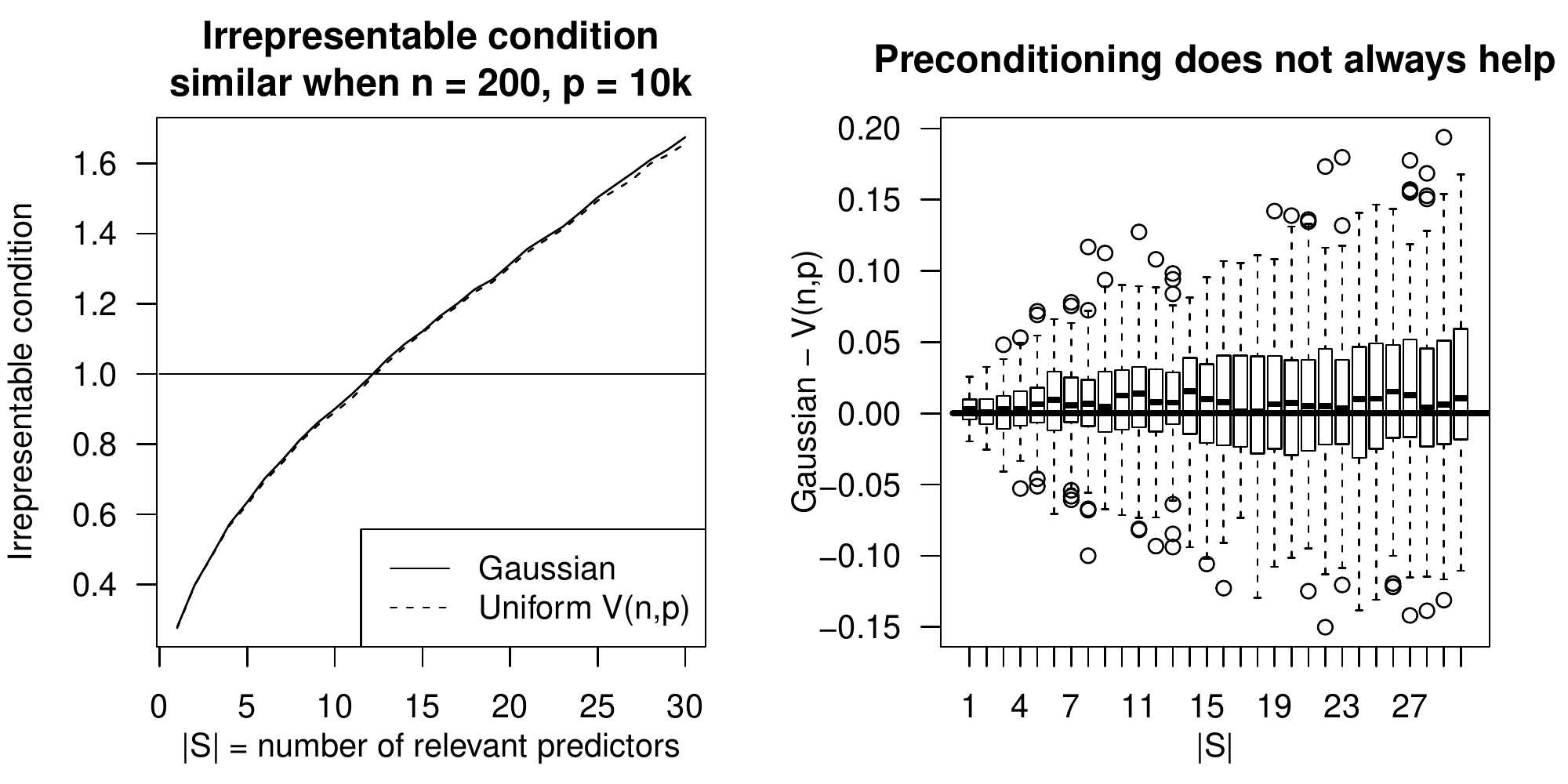} 
   \caption{The left panel displays a Monte Carlo estimate of $E(IC(\X, S))$ where $IC$ is defined in Equation \ref{ircscore}.  The solid line corresponds to random matrices $\X$ with iid $N(0,1)$ entries. The dashed line corresponds to $uniform(V(n, p))$.  The lines are nearly identical, suggesting that $uniform(V(n, p))$ design matrices behave similarly to iid Gaussian designs with respect to the irrepresentable condition.  The right panel displays the quantities $IC(\X^{(i)}, S) - IC(V^{(i)}, S)$, where $V^{(i)}$ is the projection of $\X^{(i)}$ onto $V(n, p)$.  The variability displayed in the right panel implies that preconditioning only helps for a certain type of design matrix;  it does not always help.  In this example, the columns of $\X^{(i)}$ are not highly correlated  and sometimes preconditioning decreases $IC(\cdot, S)$ and sometimes it increases $IC(\cdot, S)$. }
   \label{fig:ircforvnp}
\end{figure}

This simulation sampled 100 matrices $\X^{(i)} \in \R^{200 \times 10,000}$with iid $N(0,1)$ entries, with $i = 1, \dots, 100$.  Each of these matrices $\X^{(i)}$ were then projected onto $V(n, p)$ with the Puffer Transformation; $V^{(i)} = \pre^{(i)} \X^{(i)}$, where $\pre^{(i)}$ is the Puffer Transformation for $\X^{(i)}$.  By Proposition \ref{unifGaussian}, $V^{(i)} \sim uniform(V(n, p))$.  Before doing any of the remaining calculations, each matrix was centered to have mean zero and scaled to have standard deviation one.  For each of these (centered and scaled) matrices $\X^{(i)}$ and $V^{(i)}$, the values $IC(\cdot, S)$ are computed for $S = \{1, \dots, q\}$ for $q  = \{1, \dots, 30\}$.  (So, the matrices $\X^{(i)}$ and $V^{(i)}$ are recycled 30 times, once for each value of $s$.) The left panel of Figure \ref{fig:ircforvnp} plots the average $IC(\cdot, S)$ for each distribution, for each value of $s$.  The solid line corresponds to the $\X^{(i)}$'s, and the dashed line corresponds to the $V^{(i)}$'s.  These lines are nearly identical. %average scores for the $\X^{(i)}$'s are slightly larger than the average scores for the $V^{(i)}$'s.  
This suggests that $uniform(V(n, p))$ and the iid Gaussian design perform similarly with respect to the irrepresentable condition.

The right plot in Figure \ref{fig:ircforvnp} displays the quantities $IC(\X^{(i)}, S) - IC(V^{(i)}, S)$ as a function of $s$. The boxplots illustrate that, even between the paired samples $\X^{(i)}$ and $V^{(i)}$, there is significant variation in $IC(\cdot, S)$.  Further, there are several cases where it is negative, implying
\[IC(\X^{(i)}, S) < IC(V^{(i)}, S).\]  
In these cases, preconditioning makes $IC(\cdot, S)$ larger (in addition to make the noise terms dependent).  This suggests that only a certain type of data can benefit from preconditioning.  

%In this simulation, each sample from $\mbox{ uniform }V(n,p)$ is the projection of an iid $N(0,1)$ matrix.  In fact, this simulation pairs these matrices.  The right panel of Figure \ref{fig:ircforvnp} displays the difference in their paired irrepresentable condition scores (from Equation \ref{ircscore}).  

\section{Relationship to other data analysis techniques} \label{related}
%Left multiplying $\X$ by a preconditioning matrix can improve the performance of the Lasso.  However, 
This final section identifies four statistical techniques that incidentally precondition the design matrix.  Just as a good preconditioning matrix can improve the conditioning of $\X$, a bad preconditioning matrix can worsen the conditioning of $\X$.  Any processing or preprocessing step which is equivalent to multiplying $\X$ by another matrix could potentially lead to an ill conditioned design matrix.   It is an area for future research to assess if/how these methods affect the conditioning and if/how these issues cascade into estimation performance.  

%Several other algorithms and data analysis techniques left multiply the equation $Y = \X \truebeta + \epsilon$  by a matrix.  There are several other data analysis techniques which convert $\X \truebeta$ into $ (\X T^{-1})(T \truebeta)$.  In these techniques, the goal is not to counteract the consequences of an ill conditioned design matrix and, in fact,  the specific transformations have the potential to  worsen the conditioning and worsen the estimation.

\subsubsection*{Bootstrapping}  In the most common form of the bootstrap, the ``non-parametric" bootstrap, one samples pairs of observations $(Y_i, x_i)$ with replacement from all $n$ observations \citep{efron1993introduction}.  Each bootstrap sampled data set is equivalent to left multiplying the regression equation by a diagonal matrix with a random multinomial vector down the diagonal.  This notion is generalized in the weighted bootstrap \citep{mason1992rank}, which has been suggested for high dimensional regression \citep{arlot2010some}.  The weighted bootstrap replaces the multinomial vector of random variables with any sequence of exchangeable random variables.  In either case, such a diagonal matrix will make the singular values of the design matrix more dispersed, which could lead to potential problems in satisfying the irrepresentable condition.  If $\X$ satisfies the irrepresentable condition with a constant $\eta$, then a large proportion of the the bootstrap samples of $\X$ might not satisfy the irrepresentable condition.  Or, if they satisfy the irrepresentable condition, it might be with a reduced value of $\eta$, severely weakening their sign estimation performance. \cite{el2010high} encountered a similar problem when using the bootstrap to study the smallest eigenvalue of the sample covariance matrix.  

\subsubsection*{Generalized least squares} In classical linear regression, 
\[Y = \X \truebeta + \epsilon\]
with $\X \in \R^{n \times p}$ and $p<<n$, if the error terms have expectation zero and covariance $\Sigma_\epsilon = \sigma^2 I$, then the ordinary least squares (OLS) estimator is the best linear unbiased estimator.  If the covariance of the noise vector is a known matrix $\Sigma_\epsilon$ which is not proportional to the identity matrix, \cite{Aitken} proposed what is commonly called generalized least squares (GLS).  In GLS, the regression equation is preconditioned:
\[\Sigma_\epsilon^{-1/2}Y = \Sigma_\epsilon^{-1/2} \X \truebeta + \Sigma_\epsilon^{-1/2} \epsilon.\]
This ensures that the covariance of the noise term is proportional to the identity matrix, cov($\Sigma_\epsilon^{-1/2} \epsilon) = I$.  Applying OLS to the transformed equation gives a best linear unbiased estimator. 

%Where OLS minimizes $(Y - \X b)^T (Y - \X b)$, GLS minimizes $(Y - \X b)^T \Sigma_\epsilon^{-1}(Y - \X b)$. This is equivalent to preconditioning with the matrix  (which  and then applying OLS to the transformed regression equation.  Because the noise terms now have covariance equal to the identity matrix, GLS provides a best linear unbiased estimator.  

\cite{huang2010spatial} proposed $\ell_1$ penalized GLS for spatial regression.  In effect, they preconditioned the regression equation with $\Sigma_\epsilon^{-1/2}$.  This is potentially inadvisable.  The harm from an ill conditioned design matrix potentially overwhelms the gains provided by uncorrelated error term.  Similarly, if a data set has heteroskedastic noise, it is inadvisable to reweight the matrix to make the noise terms homoscedastic.  This was observed in Simulation 3 of \cite{jia}\footnote{It was the research in \cite{jia} that prompted the current paper.}

\subsubsection*{Generalized linear models}
In a generalized linear model, the outcome $Y_i$ is generated from a distribution in the exponential family and   $E(Y_i) = g^{-1}(x_i \truebeta)$ for some link function $g$ \citep{nelder1972generalized}.  
%Iteratively re-weighted least squares (IRLS) is the most common algorithm to estimate generalized linear models.  On each iteration of IRLS, the regression equation is preconditioned with a different diagonal matrix.  
In the original Lasso paper, \cite{tibshirani1996regression} proposed applying a Lasso penalty to generalized linear models. To estimate these models, \cite{park2007l1} and \cite{friedman2010regularization} both proposed iterative algorithms with fundamental similarities to Iteratively Reweighted Least Squares.  At every iteration of these algorithms, $\X$ and $Y$ are left multiplied by a different re-weighting matrix.  It is an area for future research to examine if the sequence of re-weighting matrices might lead to an ill conditioned design and whether there are additional algorithmic steps that might ameliorate this issue.  %There are several areas for future research in this direction.  For example, the data analyst should be able to ensure that the re-weighting matrix does not ill-condition the design matrix.  Ideally, they should have a (potentially newly designed) series of re-weighting matrices that can simultaneously precondition the design matrix

\subsubsection*{Generalized Lasso}
Where the standard Lasso penalizes with $\|b\|_1$, the generalized Lasso penalizes $\|D b\|_1$ for some predefined matrix $D$ \cite{tibshirani2011solution}.  For example, $D$ may exploit some underlying geometry in the data generating mechanism.  If $D$ is invertible, then the generalized Lasso problem 
\[\min_b \|Y - \X b\|_2^2  + \lambda \|D b\|_1\]
can be transformed into a standard Lasso problem for $\theta = D b$,
\[\min_\theta \|Y - \X D^{-1} \theta \|_2^2  + \lambda \|\theta \|_1.\]
In this problem, the design matrix is right multiplied by $D^{-1}$.  In this paper, we only consider left preconditioning, where the matrix $\X$ is left multiplied by the preconditioner.  However, in numerical linear algebra, there is also right preconditioning.  Just as a bad left preconditioner can make the design matrix ill conditioned, so can a bad right preconditioner.  In practice, the statistician should understand if their matrix $D$ makes their design ill conditioned. 

%\subsubsection*{Stability selection} 
%Stability selection perturbs the lengths of the columns of the design matrix.  This is equivalent to right multiplying the design matrix by a diagonal matrix and left multiplying $\truebeta$ by the inverse of that matrix.  This is another example of right preconditioning.  Because the preconditioner is a diagonal matrix, the sparsity pattern of $\truebeta$ does not change.   A matrix and its transpose have the same nonzero singular values.  So, re-weighting the columns instead of the rows has a similar effect on the spectrum.  

\section{Discussion} \label{disc}

The information technology revolution  has brought on a new type of data, where the number of predictors, or covariates, is far greater than the number of observations.  To make statistical inferences in this regime, one needs to assume that the ``truth" lies in some low dimensional space. This paper addresses the Lasso for sparse regression.  To achieve sign consistency, the Lasso requires a stringent irrepresentable condition.  To avoid the irrepresentable condition, \cite{fan2001variable,  zou2006adaptive, zhang2010nearly}  suggest alternative penalty functions, \cite{xiong2011orthogonalizing} proposed an EM algorithm, and others, including \cite{greenshtein2004persistence} and \cite{shao2012estimation}, have described that alternative forms of consistency that do not require stringent assumptions on the design matrix.  In this paper, we show that it is not necessary to rely on nonconvex penalties, nor is it necessary to move to alternative forms of consistency.  We show that preconditioning has the potential to circumvent the irrepresentable condition in several settings.  This means that a preprocessing step can make the Lasso sign consistent. Furthermore, this preprocessing step is easy to implement and it is motivated by a wide body of research in numerical linear algebra.

The preconditioning described in this paper left multiplies the design matrix $\X$ and the response $Y$ by a matrix $\pre = UD^{-1}U'$, where $U$ and $D$ come from the SVD of $\X = UDV'$.  This preprocessing step makes the columns of the design matrix less correlated; while the original design matrix $\X$ might fail the irrepresentable condition, the new design matrix $\pre \X$ can satisfy it. In low dimensions, our preconditioner, the Puffer Transformation, ensures that the design matrix always satisfies the irrepresentable condition.  In high dimensions, the Puffer Transformation projects the design matrix onto the Stiefel manifold, and Theorem \ref{unifTheorem} shows that most matrices on the Stiefel manifold satisfy the irrepresentable condition. In our simulation settings, the Puffer Transformation drastically improves the Lasso's estimation performance, particularly in high dimensions.  We believe that this opens the door to several other important questions (theoretical, methodological, and applied) on how preconditioning can aid sparse high dimensional inference.  We have focused on preconditioning and the sign consistency of the Lasso.  However, preconditioning also has the potential to benefit several other areas as well, including Basis Pursuit and the Restricted Isometry Principle \citep{chen1994basis, candes2008restricted} and 
forms of $\ell_2$ consistency for the Lasso and the restricted eigenvalue condition \citep{bickel2009simultaneous}.

This is the first paper to demonstrate how preconditioning the standard linear regression equation can circumvent the stringent irrepresentable condition.  This represents a computationally straightforward fix for the Lasso inspired by an extensive numerical linear algebra literature.  The algorithm easily extends to high dimensions and, in our simulations, demonstrates a selection advantage and improved $\ell_2$ performance over previous techniques in very high dimensions.  
%%In many situations, this can circumvent the stringent assumptions of the Lasso.  
%When $n\ge p$,  the columns of $\pre \X$ are orthogonal, trivially satisfying the assumption.  When $n \le p$, Section \ref{highdim} gives several scenarios when preprocessing can help to circumvent the stringent assumptions.  

%and prove that in the low dimensional setting $(p\le n)$, these conditions disappear.  In the high dimensional setting $(p>n)$, the preconditioned design matrix often satisfies the stringent Lasso assumptions, even if the original matrix did not.  In this preprocessing step the entire regression equation is left multiplied by a preconditioner.  

\appendix
\begin{center}
{\large Appendix}
\end{center}
We prove our theorems in the appendix.

\section{Low dimensional settings}

We first give a well-known result that makes sure the Lasso exactly recovers the sparse pattern of $\truebeta$, that is
$\hbet =_s \truebeta$. The following Lemma gives necessary
and sufficient conditions for $\mbox{sign}(\hbet) =
\mbox{sign}(\truebeta)$.  \cite{wainwright2009} gives
these conditions that follow from KKT conditions.

\begin{lemma}{\label{KKT}}For linear model $Y=\X\truebeta+\epsilon$,
assume that the matrix $\Xa^T\Xa$ is invertible. Then for any given
$\lam>0$ and any noise term $\e\in \R^n$, there exists a Lasso
estimate $\hbet$ described in Equation (\ref{lasso}) which satisfies $\hbet=_s\truebeta$, if and only if
the following two conditions hold
\begin{equation}
\left|\Xb^T\Xa(\Xa^T\Xa)^{-1}\left[\Xa^T\e-\lam
\sign(\betA)\right]-\Xb^T\e\right|\leq \lam, \label{R1}
\end{equation}
\begin{equation}
sign\left(\betA+(\Xa^T\Xa)^{-1}\left[\Xa^T\e-\lam
sign(\betA)\right]\right)=\sign(\betA), \label{R2}
\end{equation}
where the vector inequality and equality are taken elementwise.
Moreover, if the inequality (\ref{R1}) holds strictly, then
$$\hat{\beta}=(\hbetA,0)$$ is the unique optimal solution to the Lasso problem in Equation (\ref{lasso}), where
\begin{equation}
\hbetA=\betA+(\Xa^T\Xa)^{-1}\left[\Xa^T\e-\lam
\sign(\truebeta)\right]. \label{R3}
\end{equation}
\end{lemma}

{\textbf{Remarks.} } As in \cite{wainwright2009}, we state sufficient conditions for
(\ref{R1}) and (\ref{R2}). Define
$$\overrightarrow{b}=\sign(\betA),$$ and denote by $e_i$ the vector
with $1$ in the $i$th position and zeroes elsewhere. Define
$$U_i=e_i^T(\Xa^T\Xa)^{-1}\left[\Xa^T\e-\lam\overrightarrow{b}\right],$$
$$V_j=X_j^T\left\{\Xa(\Xa^T\Xa)^{-1}\lam\overrightarrow{b}-\left[\Xa(\Xa^T\Xa)^{-1}\Xa^T-I)\right] {\e}\right\}.$$
By rearranging terms, it is easy to see that (\ref{R1}) holds
strictly if and only if
\begin{equation}
\mathcal{M}(V)=\left\{\max_{j\in S^c} |V_j|< \lam\right\} \label{C1}
\end{equation} holds. If we define $\minbeta=
\min_{j\in S}|\truebetaj|$ (recall that $S=\{j: \truebetaj\neq 0\}$
is the sparsity index), then the event
\begin{equation}
\mathcal{M}(U)=\left\{\max_{i \in S }|U_i|<
\minbeta\right\},\label{C2}
\end{equation}
is sufficient to guarantee that condition (\ref{R2}) holds.

With the previous lemma, we now have our main result for the Lasso estimator on data from the linear model with {\it{correlated noise terms.}}  It requires some regularity conditions.
Suppose the irrepresentable condition holds.
That is, for some constant $\eta \in (0,1]$,
\begin{equation}
\left\|\Xb^T\Xa\left(\Xa^T\Xa\right)^{-1}\overrightarrow{b}\right\|_{\infty}\leq
1-\eta. \label{ICap}
\end{equation}
In addition,
assume \begin{equation}
C_{\min} = \Lambda_{\min}\left(\Xa^T\Xa\right) > 0.
\label{min.eigen}
\end{equation}
where $\Lambda_{\min}$ denotes the minimal eigenvalue, and

\begin{equation}
\max_j\| X_j\|_2 \leq 1 \label{norm1}
\end{equation}
%\begin{equation}
%\Sigma_\epsilon[i,i]\leq \tilde C_{\max}<\infty,
%\label{min.eigen.noise}
%\end{equation}

We also need the following Gaussian Comparison result for any mean zero Gaussian random vector.
\begin{lemma}
\label{GaussianComparison} For any mean zero Gaussian random vector
$(X_1,\ldots,X_n)$, and $t>0$, we have
\begin{equation}
P( \max_{1\leq i\leq n} |X_i| \ge t)\leq 2n
\exp\left\{-\frac{t^2}{2\max_i E(X_i^2)}\right\} \label{GCR}
\end{equation}
\end{lemma}

Define 
$$\Psi(\X, \truebeta, \lam) =\lam\left[ \frac{\eta}{\sqrt{C_{\min}} }+ \left\|\left(\Xa^T\Xa \right)^{-1}
\overrightarrow{b} \right\|_\infty \right].$$
\begin{lemma}{\label{thm:consis}}
Suppose that data $(\X,\Y)$ follow the linear model $Y = \X\truebeta + \epsilon$, with Gaussian noise $\epsilon\sim N(0,\Sigma_\e)$. Assume that regularity conditions (\ref{ICap}), $(\ref{min.eigen})$, and $(\ref{norm1})$
%, (\ref{min.eigen}) and (\ref{min.eigen.noise})
hold. 
If $\lam$ satisfies
$$\minbeta>  \Psi(\X, \truebeta, \lam),$$
then with probability greater than
$$1-2\p
\exp\left\{-\frac{\lam^2\eta^2}{2\Lambda_{\max}
(\Sigma_\e)}\right\},$$
the Lasso has a unique solution $\hbet$ with $\hbet =_s \truebeta$.
\end{lemma}

\begin{proof}
This proof is divided into two parts. First we analyze the
 probability of event $\mathcal{M}(V)$, and then we
analyze the event of $\mathcal{M}(U)$.

\textbf{Analysis of $\mathcal{M}(V):$} Note from (\ref{C1}) that
$\mathcal{M}(V)$ holds if and only if    $\frac{\max_{j \in S^c}
|V_j|}{\lam}< 1$. Each random variable $V_j$ is Gaussian with mean
$$\mu_j=\lam X_j^T\Xa(\Xa^T\Xa)^{-1}\overrightarrow{b}.$$

Define
$\tilde{V}_j=X_j^T\left[I-\Xa(\Xa^T\Xa)^{-1}\Xa^T\right] {\e}$,
then $V_j=\mu_j+\tilde{V}_j$. Using the irrepresentable condition (\ref{ICap}), we have
$|\mu_j|\leq (1-\eta)\lam$ for all $j \in S^c$, from which we obtain
that
$$\frac{1}{\lam} \max_{j\in S^c} |\tilde{V}_j|< \eta \Rightarrow \frac{\max_{j \in
S^c} |V_j|}{\lam}< 1.$$
Recall $s = |S|$.  By the Gaussian comparison result (\ref{GCR}) stated in Lemma
\ref{GaussianComparison}, we have
$$P\left[\frac{1}{\lam}\max_{j\in
S^c}|\tilde{V}_j|\geq \eta\right]\leq
2(p-s)\exp\{-\frac{\lam^2\eta^2}{2\max_{j\in S^c}{E(\tilde
V_j^2)}}\}.$$ Since
$$E(\tilde{V}_j^2)=X_j^TH \Sigma_e HX_j,$$
where $H=I-\Xa(\Xa^T\Xa)^{-1}\Xa^T$ which has maximum eigenvalue
equal to $1$. An operator bound yields
$$E(\tilde{V}_j^2)\leq \Lambda_{\max}
(\Sigma_\e)\|X_j\|_2^2 \leq \Lambda_{\max}
(\Sigma_\e).$$
Therefore, 
\begin{eqnarray*}
P\left[\frac{1}{\lam}\max_j|\tilde{V}_j|\geq \eta\right]&\leq&
2(p-s)\exp\left\{-\frac{\lam^2\eta^2}{2\Lambda_{\max}
(\Sigma_\e)}\right\}.
\end{eqnarray*}
So,
\begin{eqnarray*}
P\left[\frac{1}{\lam}\max_j|V_j|< 1\right]&\geq& 1-
P\left[\frac{1}{\lam}\max_j|\tilde{V}_j|\geq \eta\right]\\&\geq&
1-2(p-s)\exp\left\{-\frac{\lam^2\eta^2}{2\Lambda_{\max}
(\Sigma_\e)}\right\}.
\end{eqnarray*}

\textbf{Analysis of $\mathcal{M}(U):$}

$$\max_i |U_i|\leq \|(\Xa^T\Xa)^{-1} \Xa^T\e\|_{\infty}+\lam\|( \Xa^T\Xa)^{-1}\overrightarrow{b}\|_{\infty}$$
Define $Z_i:=e_i^T(\Xa^T\Xa)^{-1}\Xa^T\e.$
Each $Z_i$ is a normal Gaussian with mean $0$ and variance
\begin{eqnarray*}
var(Z_i)&=&e_i^T(\Xa^T\Xa)^{-1}\Xa^T
\Sigma_\e \Xa(\Xa^T\Xa)^{-1}e_i\\
&\leq& \Lambda_{\max}(\Sigma_\e) e_i^T(\Xa^T\Xa)^{-1}\Xa^T
 \Xa(\Xa^T\Xa)^{-1}e_i\\
&\leq& \frac{\Lambda_{\max}(\Sigma_\e)}{C_{\min}}.
\end{eqnarray*}
So for any $t>0$, by (\ref{GCR})
$$P(\max_{i\in S} |Z_i|\ge t)\leq 2s
\exp\left\{-\frac{t^2
C_{\min}}{2 \Lambda_{\max}(\Sigma_\e)} \right\},$$ by taking
$t=\frac{\lam\eta}{\sqrt{C_{\min}}}$, we have
$$P(\max_{i\in S} |Z_i| \ge \frac{\lam\eta}{\sqrt{C_{\min}}})\leq 2s
\exp\left\{-\frac{\lam^2\eta^2}{2\Lambda_{\max}
(\Sigma_\e)}\right\}.$$
Recall the definition of $\Psi(\X, \truebeta, \lam) =\lam\left[ \frac{\eta}{\sqrt{C_{\min}}}+ \left\|\left(\Xa^T\Xa \right)^{-1}
\overrightarrow{b} \right\|_\infty \right]$.  We now have
$$P(\max_i |U_i|\geq \Psi(\X,\truebeta,\lam))\leq 2s
\exp\left\{-\frac{\lam^2\eta^2}{2\Lambda_{\max}
(\Sigma_\e)}\right\}.$$
By condition $\minbeta>\Psi(\X,\truebeta,\lam)$, we have
$$P(\max_i |U_i|<\minbeta)\geq 1-2s
\exp\left\{-\frac{\lam^2\eta^2}{2\Lambda_{\max}
(\Sigma_\e)}\right\}.$$
At last, we have
$$P\left[\mathcal M(V) \& \  \mathcal M(U) \right]\geq
1-2\p
\exp\left\{-\frac{\lam^2\eta^2}{2\Lambda_{\max}
(\Sigma_\e)}\right\}.$$

\end{proof}

We are now ready to prove Theorem \ref{thm:fixeddim} in Section \ref{lowdim}.  For convenience, the theorem is repeated here.

\begin{theorem}
Suppose that data $(\X,\Y)$ follows the linear model described in Equation
(\ref{regeq}) with  iid normal noises $\e \sim N(0, \sigma^2I_n)$. Define the singular value decomposition of $X$ as $\X = UDV^T$.  Suppose that $n\geq p$ and $\X$ has rank $p$. We further assume that $\Lambda_{\min}(\frac{1}{n}X^TX) \geq \tilde C_{\min}>0$. Define the \textbf{Puffer Transformation},
$\pre  = UD^{-1}U^T.$ Let $\tilde \X = F \X, \tilde Y = F Y$ and $\tilde \e = F \e$. Define

\[\tilde \beta(\lambda) = \arg\min_b \frac{1}{2} \|\tilde Y - \tilde \X b\|_2^2 + \lambda  \| b\|_1.\]

If $\min_{j\in S}|\truebeta_j| \geq 2  \lambda$, 
then with probability greater than
$$1-2p \exp\left\{-\frac{n \lam^2\tilde C_{\min}}{2\sigma^2}\right\} $$
we have $\tilde \beta(\lambda) =_s \truebeta$.
\end{theorem}

\begin{proof}
Data after transformation $(\tilde \X, \tilde Y)$ follows the following linear model:
\[\tilde Y = \tilde \X \truebeta + \tilde e,\] with $\tilde e$ having co-variance matrix  $\Sigma_{\tilde e} = \sigma^2 F^TF = \sigma^2 UD^{-2} U^T$.

Since \[\tilde \X' \tilde \X = \X'F'F\X = [VDU^T] [U D^{-2} U^T] [UDV^T] = I_{p\times p}.\] So the irrepresentable condition (\ref{ICap}) holds with $\eta = 1$.
To apply Lemma \ref{thm:consis}, we first calculate $\Psi(\tilde \X, \truebeta, \lam)$.
\begin{eqnarray*}
\Psi(\tilde \X, \truebeta, \lam) =\lam\left[ \frac{\eta}{\sqrt{C_{\min}}\max_{j\in S^c} \|\tilde X_j\|_2}+ \left\|\left(\tilde \Xa^T\tilde \Xa \right)^{-1}
\overrightarrow{b} \right\|_\infty \right].
\end{eqnarray*}

Notice that 
\[\tilde \X' \tilde \X = I_{p\times p}.\]
So, $C_{\min} = 1$, $\|\tilde X_j\|_2 = 1$ and
$\left\|\left(\tilde \Xa^T\tilde \Xa \right)^{-1}
\overrightarrow{b} \right\|_\infty = {1}$ and, consequently,
\begin{eqnarray*}
\Psi(\tilde \X, \truebeta, \lam) &=&\lam\left[ \frac{\eta}{\sqrt{C_{\min}}}+ \left\|\left(\tilde \Xa^T\tilde \Xa \right)^{-1}
\overrightarrow{b} \right\|_\infty \right]\\
&=& 2\lambda.
\end{eqnarray*}

Now we calculate the lower bound probability:
\begin{eqnarray*}1-2\p
\exp\left\{-\frac{\lam^2\eta^2}{2\Lambda_{\max}
(\Sigma_{\tilde \e})}\right\}
\end{eqnarray*}

Notice that $\Sigma_{\tilde \e} = \sigma^2 UD^{-2} U'$. So $\Lambda_{\max}
(\Sigma_{\tilde \e}) = \frac{\sigma^2}{(\min_i D_{ii})^2}$. From $\Lambda_{\min}(\frac{1}{n}X^TX) \geq \tilde C_{\min}$, we see that
$\tilde C_{\min} \leq \Lambda_{\min}(\frac{1}{n}X^TX) = \frac{1}{n}  \Lambda_{\min}(VD^2V^T) = \frac{1}{n} \min_i (D_{ii}^2)$.
This is to say $\min_{i} D_{ii}^2 \geq n \tilde C_{\min}$.

\begin{eqnarray*}1-2\p
\exp\left\{-\frac{\lam^2\eta^2}{2[\Lambda_{\max}
(\Sigma_{\tilde \e})]}\right\}\\
=1-2p \exp\left\{-\frac{\lam^2\min_i(D_{ii}^2)}{2 \sigma^2 }\right\}\\
\geq 1-2p \exp\left\{-\frac{n\lam^2\tilde C_{\min}}{2\sigma^2}\right\} 
\end{eqnarray*}
\end{proof}

%\begin{theorem}
%Suppose that data $(\X,\Y)$ follows the linear model described at
%\eqref{regeq} with  iid normal noises $\e \sim N(0, \sigma^2I_n)$. Define the singular value decomposition of $X$ as $\X = UDV^T$.  Suppose that $p\geq n$. We further assume that $\Lambda_{\min}(\frac{1}{n}\Xa^T\Xa) \geq \tilde C_{\min}$  and ${\min_{i}}(D^2_{ii}) \geq  p d_{\min}$ with  constants $\tilde C_{\min}>0$ and $d_{\min}>0$. Define the \textbf{Puffer Transformation},
%$\pre  = UD^{-1}U^T.$ Let $\tilde \X = F \X, \tilde Y = F Y$ and $\tilde \e = F \e$. Define
%
%\[\tilde \beta(\lambda) = \arg\min_b \frac{1}{2} \|\tilde Y - \tilde \X b\|_2^2 + \lambda  \| b\|_1.\]
%
%Under the following THREE conditions, 
%\begin{enumerate}
%\item $\left\|\tilde \Xb^T\tilde\Xa\left(\tilde\Xa^T\tilde\Xa\right)^{-1}\overrightarrow{b}\right\|_{\infty}\leq 1-\eta, $
%\item $\Lambda_{\min}\left(\tilde \Xa^T\tilde \Xa\right) \geq \frac{n}{cp}$
%\item $\min_{j\in S}|\truebeta_j| \geq 2\lam \sqrt{scp/n} $ 
%\end{enumerate}
%where $c$ is some constant; we have $\tilde \beta(\lambda) =_s \truebeta$  with probability greater than
%$$1-2\p
%\exp\left\{-\frac{p\lam^2\eta^2d_{\min}}{2\sigma^2 }\right\} .$$
%\end{theorem}

Next, we prove Theorem \ref{highdimthm} in Section \ref{highdim}.  The theorem is repeated here for convenience. 
\begin{theorem}
Suppose that data $(\X,\Y)$ follows the linear model described at
(\ref{regeq}) with  iid normal noises $\e \sim N(0, \sigma^2I_n)$. Define the singular value decomposition of $X$ as $\X = UDV^T$.  Suppose that $p\geq n$. We further assume that $\Lambda_{\min}(\frac{1}{n}\Xa^T\Xa) \geq \tilde C_{\min}$  and ${\min_{i}}(D^2_{ii}) \geq  p d_{\min}$ with  constants $\tilde C_{\min}>0$ and $d_{\min}>0$. For $\pre  = UD^{-1}U^T,$ define $\tilde Y = F Y$ and $\tilde \X = F \X$. Define

\[\tilde \beta(\lambda) = \arg\min_b \frac{1}{2} \|\tilde Y - \tilde \X b\|_2^2 + \lambda  \| b\|_1.\]

Under the following three conditions, 
\begin{enumerate}
\item $\left\|\tilde \X(S^c)^T\tilde\X(S)\left(\tilde\X(S)^T\tilde\X(S)\right)^{-1}\overrightarrow{b}\right\|_{\infty}\leq 1-\eta, $
\item $\Lambda_{\min}\left(\tilde \X(S)^T\tilde \X(S)\right) \geq \frac{n}{cp}$
\item $\min_{j\in S}|\truebetaj| \geq 2\lam \sqrt{scp/n} $ 
\end{enumerate}
where $c$ is some constant; we have 
\begin{equation} \label{probbound}
P\left(\tilde \beta(\lambda) =_s \truebeta\right) > 1-2\p \exp\left\{-\frac{p\lam^2\eta^2d_{\min}}{2\sigma^2 }\right\} .
\end{equation}
\end{theorem}

% \textbf{Remarks.} From the discussion of Stiefel manifold in Section \ref{subsec:Smanifold},  if $\tilde X$ is uniformly distributed on the Stiefel manifold, Conditions (1) and (2)  hold with high probability. It is easy to see that if $\min_{j\in S}|\truebeta_j|$ is a constant, there exists some $\lambda$ such that Condition (3) holds and the probability that  $\tilde \beta(\lambda) =_s \truebeta$ goes to 1, if we further have (1) $\lambda  \sqrt{sp/n} \rightarrow 0$ and (2) $\lambda^2 p / \log p \rightarrow \infty$. A possible choice of $\lambda$ is such that $\lambda^2 = \sqrt{n\log( p )/(sp^2)}$, when $n/(s\log p) \rightarrow \infty.$

\begin{proof}
After transformation, $(\tilde \X, \tilde Y)$ follows the following linear model:
\[\tilde Y = \tilde \X \truebeta + \tilde e,\] with $\tilde e$ having co-variance matrix  $\Sigma_{\tilde e} = \sigma^2 F^TF = \sigma^2 UD^{-2} U^T$.

To apply Lemma \ref{thm:consis}, first calculate $\Psi(\tilde \X, \truebeta, \lam)$.
\begin{eqnarray*}
\Psi(\tilde \X, \truebeta, \lam) =\lam\left[ \frac{\eta}{\sqrt{C_{\min}}}+ \left\|\left(\tilde \Xa^T\tilde \Xa \right)^{-1}
\overrightarrow{b} \right\|_\infty \right],
\end{eqnarray*}
where $ C_{\min} = \Lambda_{\min}\left(\tilde \Xa^T\tilde \Xa\right) $.
%Since $\tilde \Xa^T\tilde \Xa = (F\Xa)^T(F\Xa) = \Xa^TF^TF\Xa = \Xa^T UD^{-2}U^T\Xa.$
By condition (3), we have $C_{\min} \geq \frac{n}{cp}$.
\[\left\|\left(\tilde \Xa^T\tilde \Xa \right)^{-1}
\overrightarrow{b} \right\|_\infty \leq \sqrt{\frac{s}{C_{\min}}} \leq \sqrt{scp/n}\] 
and, consequently,
\begin{eqnarray*}
\Psi(\tilde \X, \truebeta, \lam) &=&\lam\left[ \frac{\eta}{\sqrt{C_{\min}}}+ \left\|\left(\tilde \Xa^T\tilde \Xa \right)^{-1}
\overrightarrow{b} \right\|_\infty \right]\\
&\leq& \lam\left[ \frac{\eta\sqrt{cp}}{\sqrt{n} }+  \sqrt{scp/n} \right]\\
&\leq& 2\lam \sqrt{scp/n} .
\end{eqnarray*}

Now we calculate the lower bound probability:
\begin{eqnarray*}1-2\p
\exp\left\{-\frac{\lam^2\eta^2}{2\Lambda_{\max}
(\Sigma_{\tilde \e})}\right\}
\end{eqnarray*}

Notice that $\Sigma_{\tilde \e} = \sigma^2 UD^{-2} U'$. So $\Lambda_{\max}
(\Sigma_{\tilde \e}) = \frac{\sigma^2}{(\min_i D_{ii})^2}$. 

\begin{eqnarray*}1-2\p
\exp\left\{-\frac{\lam^2\eta^2}{2\Lambda_{\max}
(\Sigma_{\tilde \e})}\right\}\\
=1-2\p
\exp\left\{-\frac{\lam^2\eta^2\min_{i} D_{ii}^2}{2\sigma^2 }\right\}\\
\geq 1-2\p
\exp\left\{-\frac{p\lam^2\eta^2d_{\min}}{2\sigma^2 }\right\}\\
\end{eqnarray*}
\end{proof}

\section{Uniform distribution on the Stiefel manifold}
To prove Theorem \ref{unifTheorem}, we first present some results related to Beta distributions.

\subsection{Beta distribution}

The density function for the Beta distribution with shape parameters $\alpha>0$ and $\beta>0$ is 
\[f(x) = cx^{\alpha-1}(1-x)^{\beta-1},\]
for $x\in(0,1)$.  $c$ is a normalization constant.  If $X$ follows a Beta distribution with parameters $(\alpha,\beta)$, it is denoted $X\sim Beta(\alpha,\beta)$.

\begin{prop}
If $X\sim Beta(\alpha,\beta)$, then
\[E(X) = \frac{\alpha}{\alpha+\beta} \ \mbox{ and } \ var(X) = \frac{\alpha\beta}{(\alpha+\beta)^2(\alpha+\beta+1)}.\]
%\[var(X) = \frac{\alpha\beta}{(\alpha+\beta)^2(\alpha+\beta+1)}.\]
\end{prop}

The next two inequalities for the $\chi^2$ distributions can be found from \cite{laurent2000adaptive} (pp.\ 1352).

\begin{lemma}
Let $X$ be a $\chi^2$ distribution with $n$ degrees of freedom.  Then, for any positive $x$, we have
\[P(X - n \geq 2\sqrt{nx} + 2x) \leq \exp(-x),\]
\[P(X-n \leq -2\sqrt{nx}) \leq \exp(-x).\]
\end{lemma}

By taking $x = \sqrt{n}$, from the above inequality we immediately have
\begin{corollary}
\label{coro:chi2}
Let $X$ be a $\chi^2$ distribution with $n$ degrees of freedom. There exists a constant $c$ such that for a large enough $n$, 
\[P(|X - n| \geq c n^{3/4})\leq \exp\left(-n^{1/2}\right).\]
\end{corollary}

Corollary \ref{coro:chi2} says that $\chi^2(n) = n+o(n)$ with high probability.  These large deviation results can give concentration inequalities for  Beta distributions. This is because a Beta distribution can be expressed via $\chi^2$ distributions.
\begin{lemma}
Suppose that $X\sim \chi^2(\alpha)$, $Y\sim \chi^2(\beta)$ and $X$ is independent of $Y$, then
$\frac{X}{X+Y}\sim Beta(\frac{\alpha}{2},\frac{\beta}{2})$.
\end{lemma}
With the relationship constructed between a Beta distribution and $\chi^2$ distributions, we can have the following 
 inequalities.
\begin{theorem}
\label{thm:beta}
Suppose $Z\sim Beta(\frac{n}{2},\frac{m}{2})$. When both $m$ and $n$ are big enough, there exists some constant $c$,
 
 \begin{eqnarray*}
\Prob{Z \geq \frac{n+cn^{3/4}}{m+n - c(m+n)^{3/4}} } &\leq & \exp\{-n^{1/2} \}+ \exp\{-(m+n)^{1/2}\} 
 \end{eqnarray*}

 \begin{eqnarray*}
\Prob{ Z \leq \frac{n-cn^{3/4}}{m+n + c(m+n)^{3/4}} } 
&\leq & \exp\{-n^{1/2} \}+ \exp\{-(m+n)^{1/2}\} 
 \end{eqnarray*}
 
\end{theorem}

\begin{proof}
Let $X\sim \chi^2(n)$ and $Y\sim\chi^2(m)$ are independent. Then $\frac{X}{X+Y}$ has the same distribution as $Z$.

By Corollary \ref{coro:chi2}, we have
\[P(|X - n| > c n^{3/4}) \leq \exp\{-n^{1/2}\},\]
\[\Prob{|X+Y - (m+n)| > c (m+n)^{3/4} }\leq \exp\{-(m+n)^{1/2}\}.\]
If $\frac{X}{X+Y} \geq \frac{n+cn^{3/4}}{m+n - (m+n)^{3/4}}$, then $X > n+cn^{3/4}$ or $X+Y < m+n - (m+n)^{3/4}$.
So
\begin{eqnarray*}
\Prob{ \frac{X}{X+Y} \geq \frac{n+cn^{3/4}}{m+n - c(m+n)^{3/4}} } &\leq& \Prob{X > n+cn^{3/4}} \\
&& \ \ \ + \ \Prob{X+Y < m+n - c(m+n)^{3/4}} \\
&\leq & \exp\{-n^{1/2} \}+ \exp\{-(m+n)^{1/2}\} 
 \end{eqnarray*}
 The same way, we have
 
 \begin{eqnarray*}
\Prob{ \frac{X}{X+Y} \leq \frac{n-cn^{3/4}}{m+n + c(m+n)^{3/4}} } 
&\leq & \exp\{-n^{1/2} \}+ \exp\{-(m+n)^{1/2}\} 
 \end{eqnarray*}
\end{proof}

\begin{corollary}
\label{coro:beta_con}
Suppose $Z\sim Beta(\frac{n}{2},\frac{m}{2})$ with $m>n$ and $n/m\rightarrow c_3$ $(0 \leq c_3< 1)$. There exists a constant $c_1$ such that the following inequality holds when both $m$ and $n$ are big enough, 
 \begin{eqnarray*}
\Prob{\left|Z-\frac{n}{m+n}\right| \geq  \frac{c_1 n^{3/4}}{m+n} } &\leq & 2\exp\{-n^{1/2}\}
 \end{eqnarray*}
 \end{corollary}
 
 Corollary \ref{coro:beta_con} states that if $Z \sim Beta(\frac{n}{2},\frac{m}{2}) (\mbox{with } m>n)$, then $Z = \frac{n}{m+n} + O(\frac{n^{3/4}}{m+n})$ with high probability. 
 
 \begin{proof}
 
From Theorem \ref{thm:beta}, we only have to prove that, for any constant $c_0$
\[\left(\frac{n-c_0n^{3/4}}{m+n + c_0(m+n)^{3/4}} - \frac{n}{m+n}\right) \left/ \left(\frac{n^{3/4}}{m+n}\right)\right.  \rightarrow c_2,\] where $c_2\in \mathbb R$ is a constant with the same sign as $c_0$.

\begin{eqnarray*}
&&\left(\frac{n-c_0n^{3/4}}{m+n + c_0(m+n)^{3/4}} - \frac{n}{m+n}\right) \left/ \left(\frac{n^{3/4}}{m+n}\right)\right.\\
& =& \frac{-c_0n^{3/4}(m+n)-c_0n(m+n)^{3/4}}{[m+n + c_0(m+n)^{3/4}](m+n)}\times\frac{m+n}{n^{3/4}}\\
& =&\frac{-c_0n^{3/4}(m+n)-c_0n(m+n)^{3/4}}{[m+n + c_0(m+n)^{3/4}]}\times\frac{1}{n^{3/4}}\\
 &=&\frac{-c_0-c_0n^{1/4}(m+n)^{-1/4}}{[1 + c_0(m+n)^{-1/4}]}\\
 &\rightarrow& -c_0 -\frac{c_0c_3}{1+c_3} \mbox{ as $m,n\rightarrow \infty$}.
\end{eqnarray*}

 \end{proof}

With the previous results, we can prove results for a random vector uniformly distributed on the Stiefel manifold.
\subsection{Stiefel manifold}
\label{subsec:Smanifold}
Suppose that $V\in \mathbb \R^{n\times p}$ with $p>n$, which satisfies $VV' = I_n$ --- the rows are orthogonal.  All of these matrices $V$ form $V(n,p)$, called the Stiefel manifold \citep{downs1972orientation}. We seek to examine the properties of a matrix $V$ that is uniformly distributed on $V(n,p)$. Specifically, we show that any two columns of $V$ are nearly orthogonal.

We suppose that $V$ comes uniformly from Stiefel manifold. Let $X = [V_j,V_k] \in \mathbb \R^{n\times 2}$ be two columns of $V$. Then, from  \cite{khatri1970note}, if $p>n+1$, the marginal density of $X$ is
\[c |I_n - XX'|^{p-2-n-1},\]
where $c$ is a normalization constant. The density of $X'X = A \in \mathbb \R^{2\times 2}$ is given by
\[c |A|^{(n-3)/2}|I_2-A|^{(p-n-3)/2}, 0< A<I,\]
where $c$ is a normalization constant.
$A$ is distributed as multivariate Beta of type I.

From \cite{Khatri1965}, 

\[E(A) =    \left(\begin{array}{cc} % or pmatrix or bmatrix or Bmatrix or ...
      n/p & 0 \\
      0 & n/p \\
   \end{array}\right).\] This result shows that any two columns of $V$ are orthogonal in expectation.
   
%   \left(\begin{array}{ccc}0 & 0 & 0 \\0 & 0 & 0 \\0 & 0 & 0\end{array}\right)

In \cite{Mitra1970}, a symmetric matrix $U \sim B_k(\frac{n_1}{2},\frac{n_2}{2})$ with $\min(n_1,n_2)\geq k$ has the density 
\[ f(U) = c |U|^{(n_1-k-1)/2}|I-U|^{(n_2-k-1)/2}.\]    
So, $A$ defined above follows $B_2(\frac{n}{2},\frac{p-n}{2})$. 
From Mitra (1970), we have the following results.
\begin{lemma}[Mitra (1970)]
If $U\sim B_k(\frac{n_1}{2},\frac{n_2}{2})$, then for each fixed non-null vector $a$,
\[a'Ua/a'a\sim Beta(\frac{n_1}{2},\frac{n_2}{2}).\]
\end{lemma}

By taking $a=(1,0)', (0,1)' \mbox{ and }(\frac{\sqrt{2}}{2},\frac{\sqrt{2}}{2})'$ respectively, we have the following results:

\begin{corollary}
If $A = \left(   \begin{array}{cc} 
      A_{11} & A_{12} \\
      A_{21} & A_{22} \\
   \end{array}\right)\in\mathbb \R^{2\times 2}\sim B_2(\frac{n}{2},\frac{p-n}{2})$, we have

\begin{enumerate}
\item $A_{11}\sim Beta(\frac{n}{2},\frac{p-n}{2})$
\item $A_{22}\sim Beta(\frac{n}{2},\frac{p-n}{2})$
\item $\frac{1}{2}A_{11} + \frac{1}{2}A_{22} + A_{12}\sim Beta(\frac{n}{2},\frac{p-n}{2})$
\end{enumerate}
\end{corollary}

Now, concentration results in Corollary \ref{coro:beta_con} can bound $A_{11},A_{12}$, and $A_{22}$.  This yields an inequality to bound $V_i'V_j/(\|V_i\|_2\|V_j\|_2)$ which  describes the linear relationship between two columns $V_i$ and $V_j$.

\begin{theorem}
\label{thm:stiefel}
Suppose that $V\in \mathbb \R^{n\times p}$ with $p\gg n$ uniformly from the Stiefel manifold. For a large enough  $n$ and $p$, there exist some constants $c_1$ and $c_2$, such that for any two different columns of $V$ --- $V_j$ and $V_k$, the following results hold:
\begin{eqnarray*}
\mbox{(1)  }\Prob{|V_j'V_k| \geq \frac{2c_1 n^{3/4}}{p} } \leq 6\exp\{-n^{1/2}\}, \mbox{ and}
\end{eqnarray*}

\begin{eqnarray*}
\mbox{(2) }\Prob{\frac{|V_j'V_k|}{\|V_j\|_2\|V_k\|_2} \geq  c_2 n^{-1/4}} \leq 8\exp\{-n^{1/2}\}
\end{eqnarray*}
\end{theorem}

\begin{proof}
\textbf{We first prove (1).} 
From Corollary \ref{coro:beta_con}, for large enough $n$ and $p$,
  \begin{eqnarray*}
\Prob{\left|A_{11}-\frac{n}{p}\right| \geq  \frac{c_1 n^{3/4}}{p} } &\leq & 2\exp\{-n^{1/2}\},
 \end{eqnarray*}
 
  \begin{eqnarray*}
\Prob{\left|A_{22}-\frac{n}{p}\right| \geq  \frac{c_1 n^{3/4}}{p} } &\leq & 2\exp\{-n^{1/2}\}, \ \mbox{ and}
 \end{eqnarray*}
 
   \begin{eqnarray*}
\Prob{\left|(A_{11} + A_{22})/2 +A_{12}-\frac{n}{p}\right| \geq  \frac{c_1 n^{3/4}}{p} } &\leq & 2\exp\{-n^{1/2}\}.
 \end{eqnarray*}

Since
\begin{eqnarray*}
|A_{12}| &\leq& \left|(A_{11} + A_{22})/2 +A_{12}-\frac{n}{p}\right| +  \left|(A_{11} + A_{22})/2 -\frac{n}{p}\right|\\
&\leq& \left|(A_{11} + A_{22})/2 +A_{12}-\frac{n}{p}\right| +  \left|\frac{A_{11}} {2} -\frac{n}{2p}\right|+  \left| \frac{A_{22}}{2} -\frac{n}{2p}\right|,
\end{eqnarray*}
it implies that, if $|A_{12}| \geq \frac{2c_1 n^{3/4}}{p} $, then 
\begin{eqnarray*}
\left|A_{11}-\frac{n}{p}\right| &\geq&  \frac{c_1 n^{3/4}}{p} \ \mbox{ or} \\
\left|A_{22}-\frac{n}{p}\right| &\geq&  \frac{c_1 n^{3/4}}{p} \ \mbox{ or} \\
\left|(A_{11} + A_{22})/2 +A_{12}-\frac{n}{p}\right| &\geq&  \frac{c_1 n^{3/4}}{p}.
\end{eqnarray*}

%$\left|A_{11}-\frac{n}{p}\right| \geq  \frac{c_1 n^{3/4}}{p} $ or 
%$\left|A_{22}-\frac{n}{p}\right| \geq  \frac{c_1 n^{3/4}}{p} $ or $\left|(A_{11} + A_{22})/2 +A_{12}-\frac{n}{p}\right| \geq  \frac{c_1 n^{3/4}}{p} $.

So,
\begin{eqnarray*}
\Prob{|A_{12}| \geq \frac{2c_1 n^{3/4}}{p} } &\leq& \Prob{\left|(A_{11} + A_{22})/2 +A_{12}-\frac{n}{p}\right| \geq  \frac{c_1 n^{3/4}}{p} }\\
&+&\Prob{\left|A_{11}-\frac{n}{p}\right| \geq  \frac{c_1 n^{3/4}}{p} } + \Prob
{\left|A_{22}-\frac{n}{p}\right| \geq  \frac{c_1 n^{3/4}}{p} }\\
&\leq& 6\exp\{-n^{1/2}\}.
\end{eqnarray*}

\textbf{Now we prove (2).} 
If
\[\frac{|V_j'V_k|}{\sqrt{\|V_j\|_2^2\|V_k\|_2^2}} \geq \frac{\frac{2c_1 n^{3/4}}{p}}{\sqrt{[\frac{n}{p} - \frac{c_1n^{3/4}}{p}]^2}},\]
we must have
$|V_j'V_k| \geq \frac{2c_1 n^{3/4}}{p}$ or $\|V_j\|_2^2 \leq \frac{n}{p} - \frac{c_1n^{3/4}}{p}$ or $\|V_k\|_2^2 \leq \frac{n}{p} - \frac{c_1n^{3/4}}{p}$. 

So,
\begin{eqnarray*}
\Prob{\frac{|V_j'V_k|}{\sqrt{\|V_j\|_2^2\|V_k\|_2^2}} \geq \frac{\frac{2c_1 n^{3/4}}{p}}{{\frac{n}{p} - \frac{c_1n^{3/4}}{p}}}} &\leq& \Prob{ |V_j'V_k| \geq \frac{2c_1 n^{3/4}}{p}} \\
&& \ \ + \ \Prob{ \|V_j\|_2^2 \leq \frac{n}{p} - \frac{c_1n^{3/4}}{p} } \\
&& \ \ + \ \Prob{ \|V_k\|_2^2 \leq \frac{n}{p} - \frac{c_1n^{3/4}}{p}}\\
&\leq& 8\exp\{-n^{1/2}\}.
\end{eqnarray*}

Note that  
\[\frac{\frac{2c_1 n^{3/4}}{p}}{{\frac{n}{p} - \frac{c_1n^{3/4}}{p}}} = O(n^{-1/4}).\]
So, for large enough $n$,  there exists a constant $c_2$ such that 
\begin{eqnarray*}
\Prob{\frac{|V_j'V_k|}{\sqrt{\|V_j\|_2^2\|V_k\|_2^2}} \geq c_2 n^{-1/4}} \leq 8\exp\{-n^{1/2}\}.
\end{eqnarray*}

\end{proof}

\begin{corollary}
Suppose that $V\in \mathbb \R^{n\times p}$ is uniformly distributed on the Stiefel manifold.  After normalizing the columns of $V$ to have equal length in $\ell_2$, the irrepresentable condition holds for $s$ relevant variables 
with probability no less than
\[1- 4p(p-1)e^{-n^{1/2}},\]
as long as $n$ is large enough and $n > (2c_2)^4(2s-1)^4$.
%the irrepresentable condition holds for normalized version of $V$  
%where normalized version of $V$, denoted as $\tilde V$ is defined as
%$\tilde V_{ij} = V_{ij} / \sqrt{\sum_{i=1}^n V_{ij}^2}.$
\end{corollary}

\begin{proof}
Let $C$ be the normalized Gram matrix of $V$ defined as
\[C_{jk} = \frac{|V_j'V_k|}{\sqrt{\|V_j\|_2^2\|V_k\|_2^2}}. \]
From Theorem \ref{thm:stiefel},
\[\Prob{|C_{jk}| \geq c_2 n^{-1/4}} \leq 8\exp\{-n^{1/2}\}.
\]
Using the union bound, 
\[\Prob{\max_{jk} |C_{jk}| \geq c_2 n^{-1/4}} \leq 4p(p-1)\exp\{-n^{1/2}\}.
\]
By Corollary 2 of Zhao and Yu (2006), when $\max |C_{ij}| \leq \frac{c}{2s-1}$ for some $0<c<1$, the irrepresentable condition holds.
So, if
$c_2 n^{-1/4} \leq  \frac{1}{2(2s-1)}$, that is $n > (2c_2)^4(2s-1)^4$, we have
\[\Prob{\max_{jk} |C_{jk}| \geq \frac{1}{2(2s-1)}} \geq 1- 4p(p-1)\exp\{-n^{1/2}\}.
\]
\end{proof}

\begin{theorem}
\label{thm:ICG}
Suppose that $X\in \R^{n\times p}$ is a random matrix with each element $X_{ij}$ drawn iid from $N(0,1)$. 
Then, after normalizing the columns of $X$ to have equal length in $\ell_2$, the irrepresentable condition holds for $s$ relevant variables with probability no less than
\[1- \frac{1}{2}p(p-1) e^{-\frac{n c^2}{16(2s-1)^2}} - 3p(p-1)e^{-\frac{n}{16}},\]
for any $0<c<1$.
%, where normalized version of $X$, denoted as $\tilde X$ is defined as $\tilde X_{ij} = X_{ij} / \sqrt{\sum_{i=1}^n X_{ij}^2}.$

\end{theorem}

This implies that $n$ must grow faster than $s^2\log(p)$ for $X$ to satisfy the irrepresentable condition.
\begin{proof}
Let $C$ be the empirical correlation matrix of $X$
\[C_{jk} = \frac{1/n \sum_{\ell=1}^{n} X_{j\ell}X_{k\ell}} {\sqrt{1/n\sum_{\ell=1}^n{X_{j\ell}^2}\times1/n\sum_{\ell =1}^n{X_{k\ell}}^2}}.\] By Corollary 2. of Zhao and Yu (2006), when $\max |C_{ij}| \leq \frac{c}{2s-1}$, IC holds.
Now we try to bound $P\left(\max_{i\neq j} |C_{ij}| \leq \frac{c}{2s-1}\right)$.

Note that 
\[{1/n\sum_{\ell=1}^{n} X_{j\ell}X_{k\ell}} \mid X_{k} \sim N\left(0,1/n^2 \sum_{\ell}{X_{k\ell}^2}\right).\]

By a concentration inequality on $\chi^2$ distribution, 
\[P(1/2<\frac{\chi^2(n)}{n}<2)\geq 1- 2e^{-\frac{n}{16}}.\]

For $Z\sim N(0,2/n)$,
\[P\left({1/n\sum_{\ell=1}^{n} X_{j\ell}X_{k\ell}} > t \mid  \frac{1}{n}\sum_\ell X_{k\ell}^2 <2 \right) \leq P(Z >t) \leq e^{-\frac{nt^2}{4}}.\]
This holds because the variance increases.  So,
\begin{eqnarray*}P\left({1/n\sum_{\ell=1}^{n} X_{j\ell}X_{k\ell}} > t  \right)  &= &P\left({1/n\sum_{\ell=1}^{n} X_{j\ell}X_{k\ell}} > t \mid  \frac{1}{n}\sum_\ell X_{k\ell}^2 <2  \right)  P\left(\frac{1}{n}\sum_\ell X_{k\ell}^2 <2  \right)\\
&&+P\left({1/n\sum_{\ell=1}^{n} X_{j\ell}X_{k\ell}} > t \mid  \frac{1}{n}\sum_\ell X_{k\ell}^2 >2  \right)  P\left(\frac{1}{n}\sum_\ell X_{k\ell}^2 >2  \right)\\
&\leq&P\left({1/n\sum_{\ell=1}^{n} X_{j\ell}X_{k\ell}} > t \mid  \frac{1}{n}\sum_\ell X_{k\ell}^2 <2  \right) +  P\left(\frac{1}{n}\sum_\ell X_{k\ell}^2 >2  \right)\\
&\leq &e^{-\frac{nt^2}{4}} + 2 e^{-n/16}.
\end{eqnarray*}
Finally,
\begin{eqnarray*}
P(|C_{jk}| < a) &\geq& P\left(|1/n\sum_{\ell=1}^{n} X_{j\ell}X_{k\ell} |< a/2 , {1/n\sum_{\ell=1}^n{X_{j\ell}^2} }> 1/2 , {1/n\sum_{\ell=1}^n{X_{k\ell}^2}} > 1/2  \right)\\
&\geq& P(1/n\sum_{\ell=1}^{n} X_{j\ell}X_{k\ell} < a/2) + P({1/n\sum_{\ell=1}^n{X_{j\ell}^2} }> 1/2) \\
&& \ \ + \  P( {1/n\sum_{\ell=1}^n{X_{k\ell}^2}} > 1/2 ) - 2\\
&\geq& 1- e^{-\frac{na^2}{16}} - 2e^{-n/16} + [1- 2e^{-\frac{n}{16}}]\times2 - 2\\
&=& 1- e^{-\frac{na^2}{16}} - 6e^{-\frac{n}{16}}.
\end{eqnarray*}
Taking $a = \frac{c}{2s - 1}$,
\[P(|C_{jk}| < \frac{c}{2s - 1}) \geq 1- e^{-\frac{n c^2}{16(2s-1)^2}} - 3e^{-\frac{n}{16}}\]
and
\[P(\max_{j\ne k} |C_{jk}| < \frac{c}{2s - 1}) \geq 1- \frac{1}{2}p(p-1) e^{-\frac{n c^2}{16(2s-1)^2}} - 3p(p-1)e^{-\frac{n}{16}}.\]
\end{proof}

\newpage
\bibliographystyle{plainnat}
\bibliography{references}

\end{document}